\documentclass[12pt]{article}
\usepackage{graphicx,amsthm} 
\usepackage{tikz, makecell, colortbl, tabularray, multirow, hhline,hyperref}
\definecolor{army}{rgb}{0.0, 0.5, 0.0}
\definecolor{chestnut}{rgb}{0.8, 0.36, 0.36}
\definecolor{awesome}{rgb}{1.0, 0.13, 0.32}
\definecolor{electricindigo}{rgb}{0.44, 0.0, 1.0}
\definecolor{teal}{rgb}{0.0, 0.5, 0.5}
\definecolor{burgundy}{rgb}{0.5, 0.0, 0.13}
\definecolor{islamicgreen}{rgb}{0.0, 0.56, 0.0}
\newtheorem{condition}{Condition}
\usepackage[margin=1in]{geometry}
\usepackage[numbers]{natbib}

\usepackage{amsmath,amssymb}
\usepackage{tikz,bookmark,epsfig,rotating,colortbl}
\usepackage{multirow, anyfontsize}

\usepackage{amsthm}%
\usepackage{mathrsfs}%
\usepackage[title]{appendix}%
\usepackage{xcolor}%
\usepackage{textcomp}%
\usepackage{manyfoot}%
\usepackage{booktabs}%

\usepackage{listings,bm}
\usepackage{comment}
\usepackage{float}
\usepackage{hhline}
\usepackage{authblk}
\newtheorem{lemma}{Lemma}
\newtheorem{proposition}{Proposition}
\newtheorem{theorem}{Theorem}
\newtheorem{corollary}{Corollary}
\title{\textbf{Strong Colouring of the Qualitative Independence Hypergraph $3\text{-}QI(11,2)$}}
\author[1]{Raina Mary Thomas}
\author[2]{Yasmeen Akhtar}

\affil[1,2]{ Department of Mathematics, BITS Pilani K K Birla Goa Campus, Goa 403726, India}

\affil[1]{\texttt{p20220019@goa.bits-pilani.ac.in}}
\affil[2]{\texttt{yasmeena@goa.bits-pilani.ac.in}}

\date{}
\begin{document}

\maketitle

\begin{abstract}
We determine the strong independence number of the qualitative independence hypergraph, $3\text{-}QI(11, 2)$, using a technique that involves considering its vertices as subsets of $\{1,2, \ldots, 11\}$ and assessing them as intersecting set systems. This gives the maximum size of colour classes in any strong colouring and thus, a lower bound on the strong chromatic number of $3\text{-}QI(11,2)$. We leverage this bound along with an upper bound of the strong chromatic number of $3\text{-}QI(10,2)$, to consequently, establish that the covering array number of $3\text{-}QI(11, 2)$, $CAN(3\text{-}QI(11,2),2) = 11$ and give a sufficient condition for a hypergraph $H$ to have $CAN(H, 2)=11$.
\end{abstract}
\noindent\textbf{Keywords}: Qualitative Independence, Hypergraph, Strong Colouring, Strong Independence, Covering Array Number\\               
\noindent\textbf{MSC codes:} 05C15, 05C35, 05C65, 90C27, 05D05                       

\section{Introduction }
Hypergraphs generalize graphs by allowing edges to connect more than two vertices, thereby modeling higher-order relationships that cannot be captured by pairwise interactions alone. A hypergraph is a pair \(H=(V,E)\), where \(V=\{v_1,\ldots,v_k\}\) is the vertex set and \(E=\{e_1,\ldots,e_m\}\) is a family of nonempty subsets of \(V\), called \emph{hyperedges}, with \(\bigcup_{i=1}^{m} e_i = V\). A hypergraph is \emph{\(t\)-uniform} if every hyperedge has exactly \(t\) vertices. The special case \(t=2\) corresponds to a (simple) graph. Figure~\ref{fig:eg} shows a hypergraph with \(V=\{a,b,c,d,e\}\) and \(E=\big\{\{a,b,c\},\{c,d,e\},\{b,e\},\{a,d\},\{b,d\}\big\}\). 
For a positive integer $k$, a \emph{strong \(k\)-colouring} of a hypergraph $H$ is a partition of its vertex set into $k$ colour classes such that no two vertices belonging to the same hyperedge share the same colour. The \emph{strong chromatic number}, \(\chi_S(H)\), is the minimum \(k\) for which \(H\) admits a strong \(k\)-colouring. A set \(U\subseteq V\) is \emph{strongly independent} if no hyperedge in $H$ contains more than one vertex of \(U\). The maximum cardinality of such a set is the \emph{strong independence number}, \(\alpha_S(H)\). In graphs, these notions reduce to the classical chromatic and independence numbers.\\

Let \(n,g\) be positive integers, and let \(S\) be a set of size \(g\). A \emph{covering array on a hypergraph} \(H=(V,E)\) of size $n$ with $g$ symbols, denoted \(CA(n,H,g)\), is an \(n\times |V|\) array with entries from \(S\) and whose columns are indexed by vertices in \(V\) such that for every hyperedge \(e\in E\), the $n\times |e|$ sub-array corresponding to the vertices in $e$ contains every ordered $|e|$-tuple from $S^{|e|}$ at least once as a row. The minimum such \(n\) is the \emph{covering array number} \(CAN(H,g)\); a covering array achieving this value is called \emph{optimal}. Figure \ref{fig:eg} shows an optimal CA on $H$ of size $8$ with $2$ symbols.\\

It is evident that $CAN(H, g)\geq g^r$, where $r$ is the \emph{rank} of $H$ given by $r=\max \{|e|: e\in E\}$. However, determining \(CAN(H,g)\) is NP-complete, as shown by Seroussi and Bshouty~\cite{NPproof} via a reduction from the \(3\)-colouring problem in graphs. 
When the hypergraph $H$ models system components and their interactions, a covering array (CA) on $H$ of size $n$ yields structured test suites for detecting interaction faults using $n$ test cases, see \cite[Figures~2-4]{akhtar2017mixed}. For further details on applications of covering arrays in combinatorial testing, refer to \cite{appl1994compressing, applstevens1998efficient, drug, alanwilliams1996practical}.

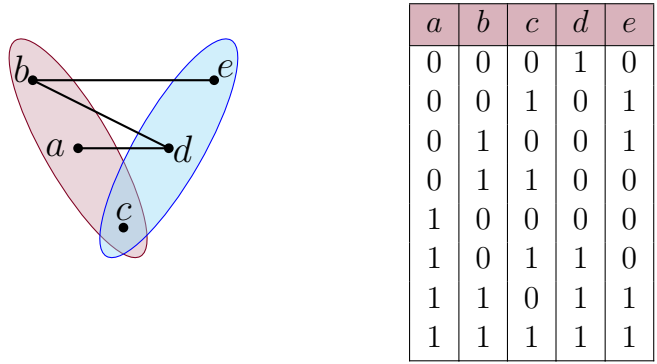
\begin{figure}[h]
    \centering
    \begin{tikzpicture}[scale=3]
   \draw[rotate around={30:(0, 0.4)}, burgundy, fill=burgundy!40, fill opacity=0.4] (0, 0.4) ellipse (0.15 cm and 0.55 cm);
				\draw[rotate around={150:(0.4, 0.4)}, blue, fill=cyan!40, fill opacity=0.4] (0.4, 0.4) ellipse (0.15 cm and 0.55 cm);
                
				\draw[black, thick] (0.4, 0.4)--(0, 0.4);
    \draw[black, thick] (0.6, 0.7)--(-0.2, 0.7);
    \draw[black, thick] (0.4, 0.4)--(-0.2, 0.7);
				
				\node[draw, black, circle,thick,minimum size=0.1cm,inner sep=0pt, fill=black] at (0.2, 0.05) {};
                \node[draw, black, circle,thick,minimum size=0.1cm,inner sep=0pt, fill=black] at (0.4, 0.4) {};
                \node[draw, black, circle,thick,minimum size=0.1cm,inner sep=0pt, fill=black] at (0.6, 0.7) {};
                \node[draw, black, circle,thick,minimum size=0.1cm,inner sep=0pt, fill=black] at (0, 0.4) {};
                \node[draw, black, circle,thick,minimum size=0.1cm,inner sep=0pt, fill=black] at (-0.2, 0.7) {};
                \node at (-0.1, 0.4) {\large $a$};
                \node at (0.2, 0.12) {\large $c$};
                \node at (0.46, 0.4) {\large $d$};
                \node at (-0.25, 0.75) {\large $b$};
                \node at (0.65, 0.75) {\large $e$};

    \node at (2,0.25){
     \begin{tblr}{
					colspec = {|c|c|c|c|c|},
    row{1} = {burgundy!30!},
    stretch = 0.7
    }\hline
							$a$ & $b$ & $c$ & $d$ & $e$\\
							\hline
							$0$ & 0 & 0 & 1 & 0\\
							0 & 0 & 1 & 0 & 1\\
							0 & 1 & 0 & 0 & 1\\
							0 & 1 & 1 & 0 & 0\\
							1 & 0 & 0 & 0 & 0\\
							1 & 0 & 1 & 1 & 0\\
							1 & 1 & 0 & 1 & 1\\
							1 & 1 & 1 & 1 & 1\\
							\hline
					\end{tblr}};
        \end{tikzpicture}	            
    \caption{A hypergraph $H$ and an optimal $CA(8, H, 2)$.}
    \label{fig:eg}
\end{figure}

A class of hypergraphs known as \emph{qualitative independence (QI) hypergraphs}, which draw on the notion of qualitative independence from extremal set theory, provides a key characterization for determining when a covering array (CA) of a given size $n$ can exist on a hypergraph. Before defining these hypergraphs, we recall the notion of qualitatively independent partitions from \cite{renyi,kleitman1973families, poljak1983qualitatively, raina}. Let $R$ be an $n$-element set. Consider a family of partitions of $R$ into $g$ classes, $\mathcal{P} = \{R_1, R_2, \ldots, R_k\}$, where $R_j = (M^j_1, M^j_2, \ldots, M^j_g)$ for $j = 1, 2, \ldots, k$. The set $\mathcal{P}$ is said to be \emph{$t$-qualitatively independent} if, for every $t$ of the partitions $R_{i_1}, R_{i_2}, \ldots, R_{i_t}$, and for every choice of classes $M^{i_j}_s$ (with $1 \leq j \leq t, 1 \leq s \leq g$), $\displaystyle \bigcap_{j=1}^{t}M^{i_j}_s \neq \emptyset$, for all choices of $s$ and distinct choices of $i_j$. Each of these partitions can be represented in terms of vectors, by defining a vector $v$ corresponding to $R_j = (M^j_1, M^j_2, \ldots, M^j_g)$ by $v[i]=k$ if $i \in M^j_k, 1 \leq i \leq n, 1 \leq k \leq g$. So, equivalently, a set of vectors of length $n$ and entries from $\{1, 2, \ldots, g\}$ is said to be \emph{$t$-qualitatively independent} if for any $t$-subset $\{v_1, v_2, \ldots , v_t\}$ and for every $g^t$ of the $g$-ary $t$-tuples say $(g_1, g_2, \ldots, g_t)$, there exists an index $j$ such that $v_i[j] = g_i$ for $i=1, 2, \ldots, t$. The column vectors in a CA on a hypergraph, corresponding to the vertices in a hyperedge, are qualitatively independent; therefore, CAs are also known as \emph{qualitatively independent families}. QI hypergraphs defined in \cite{raina, raaphorst2013thesisvariable} encapsulate this property by incorporating the hypergraph structure. Let $n$, $g$, and $t$ be positive integers such that $t \geq 2$, $n \geq g^t$. A \emph{$t$-qualitative independence hypergraph}, $t\text{-}QI(n, g) = (V, E)$, is a $t$-uniform hypergraph, where $V$ is the set of all $g$-partitions of an $n$-set with the property that every class of the partition has size at least $g^{t-1}$ (or equivalently, set of all vectors of length $n$ with entries from $S$ and with the property that every alphabet appears at least $g^{t-1}$ times and the first appearance of every alphabet is in the lexicographic order). In this hypergraph, a set of $t$ vertices forms a hyperedge when the corresponding partitions are $t$-qualitatively independent. For example, the graph $2\text{-}QI(4,2)$ has vertices $v_1= \{\{1,2\}\{3,4\}\}, v_2 = \{\{1,3\}\{2,4\}\}, v_3= \{\{1,4\},\{2,3\}\}$ and has edges between each pair of these vertices, so it is equivalent to the complete graph on three vertices, $K_3$.\\

Adjacent vertices in $t\text{-}QI(n, 2)$ form a $(t-1)$-intersecting family. 
For integers $n\geq k \geq t \geq 1$, a \emph{$t$-intersecting family}, $ \mathcal {F}$, is a collection of subsets of $[n]= \{1,\dots ,n\}$ of size $k$ such that if $A,B\in \mathcal {F}$, then $|A\cap B|\geq t$. Frankl \cite{frankl1978erdos} constructed the $t$-intersecting families
\[\mathcal{F}_{\{n,k,t,r\}}=\{A\subseteq \{1,\dots ,n\}:|A|=k{\text{ and }}|A\cap \{1,\dots ,t+2r\}|\geq t+r\}.\]
The maximum size of $t$-intersecting families has been a question of interest in the field of extremal set theory and investigated in the literature \cite{extremalsettheory, Zakharov_2024}. Given below is an extension of the celebrated Erd$\mathrm{\ddot{o}}$s-Ko-Rado theorem \cite{erdos1961intersection}.\\

\begin{theorem}[Ahlswede-Khachatrian \cite{ahlswede1996complete}]
\label{AK ekr}
If $\mathcal {F}$ is $t$-intersecting family of $k$-subsets of $[n]$ and 
$(k-t+1)(2+{\tfrac {t-1}{r+1}})\leq n<(k-t+1)(2+{\tfrac {t-1}{r}})$, then
$ |\mathcal {F}|\leq \displaystyle\max _{r\colon t+2r\leq n}|\mathcal {F}_{\{n,k,t,r\}}|.$
\end{theorem}

Raaphorst et al. \cite{raaphorst2018variable} showed that for a $t$-uniform hypergraph $H$ and positive integers $g$ and $n$, a covering array $CA(n, H, g)$ exists if and only if there is a hypergraph homomorphism from $H$ to $t\text{-}QI(n, g)$. They also proved that if there exists $H \rightarrow 3\text{-}QI(n, g)$, then $CAN(H, g) \leq CAN(t\text{-}QI(n, g), g)$. In a related result, Meagher and Stevens \cite{meagher2005covering} proved that $CAN(2\text{-}QI(n, 2), 2)=n$. Motivated by this result, we investigate the analogous question for 3-uniform qualitative independence hypergraphs: What is $CAN(3\text{-}QI(n, 2), 2)$? \\

Suppose $A$, $B$, and $C$ form a hyperedge in $3\text{-}QI(n, 2)$. Then, by definition of qualitative independence, each of $A\cap B\cap C$, $A\cap B \cap \overline{C}$, $A\cap \overline{B} \cap C$, $\overline{A}\cap \overline{B} \cap \overline{C}$, $\overline{A}\cap B \cap C$, $\overline{A}\cap \overline{B} \cap C$, $\overline{A}\cap B \cap \overline{C}$, $A\cap \overline{B} \cap \overline{C}$ must be non-empty. If $|X|<2$, where  $X$ is one of the intersections $A \cap B$, $A \cap \overline{B}$, $\overline{A} \cap B$, or $\overline{A} \cap \overline{B}$, then either $X \cap C = \emptyset$ or $X \cap \overline{C} = \emptyset$. Conversely, if $|X|\geq 2$, for any $X$, we can construct a vertex $C$ that forms a hyperedge with $A$ and $B$, by choosing one element from each intersection to include in $C$, and placing the remaining elements in $\overline{C}$. This gives rise to the following necessary and sufficient condition for adjacency in $3\text{-}QI(n, 2)$. \\

\begin{condition}\label{adjacencycondition}
A pair of distinct vertices $A$, $B$, cannot belong to a hyperedge in $3\text{-}QI(n, 2)$ if and only if at least one of the intersections $A \cap B$, $A \cap \overline{B}$, $\overline{A} \cap B$, $\overline{A} \cap \overline{B}$ has cardinality at most one.\end{condition}

\noindent The following lemma records a basic property of $(n-1)$-subsets of an $n$-element set, which will be used in the next section.\\

\begin{lemma} \label{condition}
For any $n \geq 4$, any two distinct $(n-1)$-subsets of an $n$-element set have exactly $(n-2)$ elements in common.
\end{lemma}

In this paper, we analyse the structure of $3\text{-}QI(11,2)$, whose vertices are the $2$-partitions of $[11]=\{1, 2, \ldots, 11\}$ with part sizes $(4,7)$ or $(5,6)$, represented by their smaller parts, that is, $4$- or $5$-subsets of $[11]$ (written in string form with $10$ denoted by $0$ and $11$ by $x$). Hence, $3\text{-}QI(11,2)$ has $\binom{11}{5}+\binom{11}{4}=792$ vertices. A hyperedge in $3\text{-}QI(11,2)$ consists of three vertices whose vector representations jointly realize all the $2^3=8$ possible $3$-tuples. Fixing the first coordinate to be $0$, the remaining tuples are distributed among the other $10$ coordinates. If exactly one coordinate contributes the tuple $000$, this can be done in $S(10,7)\cdot 7!$ ways; otherwise, all the $8$ tuples are distributed in $S(10,8)\cdot 8!$ ways. Dividing by $3!$ to account for permutations of vertices, the number of hyperedges is
\[
\frac{1}{3!}\left(S(10,8)\cdot 8! + S(10,7)\cdot 7!\right)=9{,}979{,}200,
\]
where $S(n,k)$ denotes the Stirling numbers of the second kind. We obtain a lower bound on the strong chromatic number of $3\text{-}QI(11,2)$ and determine its strong independence number via the maximum size of a colour class in a strong colouring. As a consequence, we show that $CAN(3\text{-}QI(11,2),2)=11$. Together with earlier results~\cite{raina}, this implies that $CAN(3\text{-}QI(n,2),2)=n$ for $n=8,9,10,11,12$. We conjecture that this equality holds for all $n>12$. \\

The remainder of the paper is organized as follows. Section \ref{sec3} explores bounds on the size of the colour classes in any strong colouring of $3\text{-}QI(11,2)$. These bounds lead to the strong independence number and a lower bound on the strong chromatic number of $3\text{-}QI(11,2)$ given in Section \ref{sec4}. We then establish that the covering array number, $CAN(3\text{-}QI(11,2),2)=11$. We also give a family of $3$-uniform hypergraphs $H$ such that $CAN(H,2)=11$ and then conclude with a discussion in Section \ref{sec5}. 

\section{Upper Bounds on the Size of Colour Classes}
\label{sec3}

A \emph{colour class} in a strong colouring is a set of vertices that are assigned the same colour. Thus, every colour class is a strongly independent set. In this section, we fix an arbitrary vertex in $3\text{-}QI(11,2)$ and look at the maximum size of a colour class containing it.\\

\begin{lemma}
  \label{14_1}
      Let $A=abcde$ be a vertex of $3\text{-}QI(11,2)=(V, E)$. Consider a family $\mathcal{F}=\{B\in V: |A \cap B|=0, |B|=5\}$. Then, $|\mathcal{F}|=6$ and together with $A$, they can form a colour class in $3\text{-}QI(11, 2).$
  \end{lemma}
  \begin{proof}
  There are six ways to choose five elements from $\overline{A}=fghijk$ to form a set in $\mathcal{F}$: $B_1=fghij,\; B_2=fghik,\; B_3=fghjk,\; B_4=fgijk,\; B_5=fhijk,\; B_6=ghijk$. By Lemma~\ref{condition}, for $i\neq j$, $|B_i\cap B_j|=4$. Hence, one of $B_i\cap B_j$, $B_i\cap\overline{B_j}$, $\overline{B_i}\cap B_j$, or $\overline{B_i}\cap\overline{B_j}$ has at most one element. By Condition~1, all the $B_i$s together with $A$ form a valid colour class in $3\text{-}QI(11,2)$.
  \end{proof}
\begin{lemma}
  \label{14_2}
      Let $A=abcde$ be a vertex of $3\text{-}QI(11,2)=(V, E)$. Consider a family $\mathcal{F}=\{B\in V: |A \cap B|=1, |B|=5\}$. Then, at most five from $\mathcal{F}$ along with $A$ can form a colour class in $3\text{-}QI(11, 2).$
  \end{lemma}
  \begin{proof}
 For any $B \in\mathcal{F}$, the intersecting element is chosen from $A$ and the remaining from $\overline{A}=fghijk$, so $|\mathcal{F}|=\binom{5}{1} \binom{6}{4}=75$. For $x\in\{a,b,c,d,e\}$, consider the subfamily $\mathcal{F}_x=\{B\in\mathcal{F}:A\cap B=\{x\}\}$. Without loss of generality, suppose $B_1=afghi\in\mathcal{F}$. Then, two cases arise: (i) $\mathcal{F}\subseteq \mathcal{F}_a$; (ii) $\mathcal{F}\cap \mathcal{F}_x\neq\emptyset$ for at least one $x\neq a$.\\

 If $\mathcal{F}\subseteq \mathcal{F}_a$, then by Condition~1, for another vertex $B_2\in\mathcal{F}_a$ to belong to the same colour class as $B_1$, we must have $|B_1\cap B_2|\in\{0,1,4\}$. Since $a\in B_1\cap B_2$ and the remaining four digits are chosen from $\overline{A}$, the pigeonhole principle implies that at least two of these elements coincide, yielding $|B_1\cap B_2|\ge 3$. Hence $|B_1\cap B_2|=4$. Consequently, $B_2$ must contain exactly three digits from $fghi$, namely, one of $fgh$, $ghi$, $fhi$, or $fgi$. The choices of $B_2$ are arranged in the array shown in Figure~\ref{14_2Table}.
 
\begin{figure}[h]
\centering
    \begin{tabular}{|c|c|c|c|}
    \hline
        $afghj$ & $aghij$ & $afhij$&$afgij$\\
        \hline
        $afghk$ & $aghik$&$afhik$&$afgik$\\
        \hline   
    \end{tabular}  
    \caption{Vertex choices in $\mathcal{F}_a$ to be in a colour class with $A=abcde$ and $B_1=afghi$ in Lemma \ref{14_2}.}
    \label{14_2Table}
    \end{figure}
    
    Observe that any two sets from the same row of the array in Figure~\ref{14_2Table} intersect in $4$ digits; hence, by Condition~1, the sets in a row together with $B_1$ form a colour class of size $5$. Moreover, any two sets from different rows and different columns intersect in $3$ digits, violating Condition~1, and therefore cannot belong to the same colour class. Thus, choosing a set eliminates all options in the other row except the one in the same column. Consequently, selecting two sets from the same column yields a colour class of size $3$ that includes $B_1$. Hence, at most five sets from $\mathcal{F}_x$ (for any $x$) can belong to the same colour class. See Figure~\ref{tree1}. \\
    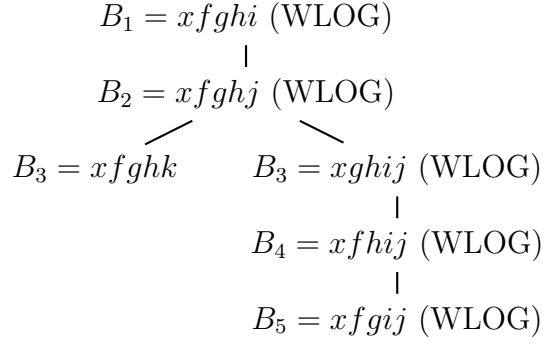
\begin{figure}[h]
    \centering
    \begin{tikzpicture}[
     edge from parent/.style={draw, black, thick},
     sibling distance=4cm,
     level distance=1cm]
     \node (root) {$B_1=xfghi$ (WLOG)}
      child {node {$B_2=xfghj$ (WLOG)}
        child {node {$B_3=xfghk$}}
          child {node {$B_3=xghij$ (WLOG)}
            child {node {$B_4=xfhij$ (WLOG)}
              child {node {$B_5=xfgij$ (WLOG)}}
            }
           }
          };
    \end{tikzpicture}
    \caption{Possibilities of forming colour classes in $\mathcal{F}_x$, in Lemma \ref{14_2}.}
    \label{tree1}
    \end{figure}
  
    Next, suppose $\mathcal{F}\cap \mathcal{F}_x\neq\emptyset$ for some $x\neq a$. Fix $B_1=afghi$ and choose $B_2\in\mathcal{F}_x$ with $x\neq a$. By Condition~1, $B_2$ must intersect $fghi$ in $0$, $1$, or $4$ digits. By the pigeonhole principle, only an intersection of $4$ digits is possible. Hence, the possible choices are $bfghi$, $cfghi$, $dfghi$, and $efghi$. Together with $B_1$, these five sets can be in a colour class containing $A$. Choosing an additional set from $\mathcal{F}_a$ eliminates all options from $\mathcal{F}_x$ for $x\neq a$, and vice versa, since such pairs intersect in $3$ digits. Therefore, at most five vertices from $\mathcal{F}$ can form a colour class with $A$: either all vertices from $\mathcal{F}_x$ for some fixed $x$, or one vertex from each $\mathcal{F}_x$ sharing a common $4$-subset of $\overline{A}$. 
    \end{proof}
      
\begin{lemma}
  \label{14_3}
      Let $A=abcde$ be a vertex of $3\text{-}QI(11,2)=(V, E)$. Consider a family $\mathcal{F}=\{B\in V: |A \cap B|=4, |B|=5\}$. Then at most six from $\mathcal{F}$ along with $A$ can form a colour class in $3\text{-}QI(11, 2).$
  \end{lemma}
  \begin{proof}
       For any $B \in \mathcal{F}$, there are $\binom{5}{4}$ options to choose four digits common with $A$ and $\binom{6}{1}$ options to select the fifth digit from $\overline{A}=fghijk$ implying $|\mathcal{F}|= \binom{5}{4} \binom{6}{1} = 30$. We lay out the members of $\mathcal{F}$ as a $6 \times 5$ array in Figure \ref{14_1table} such that each column corresponds to a $4$-subset of $A$, and each row corresponds to a digit from $\overline{A}$. \\
       
       \begin{figure}[h]
       \centering
		\begin{tabular}{|c|c|c|c|c|}
  \hline
		$abcdf$ & $abdef$ & $acdef$ & $bcdef$ & $abcef$ \\
  \hline
		$abcdg$ & $abdeg$ & $acdeg$ & $bcdeg$ & $abceg$ \\
  \hline
		$abcdh$ & $abdeh$ & $acdeh$ & $bcdeh$ & $abceh$ \\
  \hline
		$abcdi$ & $abdei$ & $acdei$ & $bcdei$ & $abcei$ \\
  \hline
		$abcdj$ & $abdej$ & $acdej$ & $bcdej$ & $abcej$ \\
  \hline
		$abcdk$ & $abdek$ & $acdek$ & $bcdek$ & $abcek$ \\
  \hline
	\end{tabular}
 
 \caption{An arrangement of elements of $\mathcal{F}=\{B\in V: |A\cap B|=4, |B|=5\}$ as described in Lemma \ref{14_3}.}
 \label{14_1table}
	\end{figure} 

     Note that the entries within a column form a colour class since Condition \ref{adjacencycondition} holds. Also, by Lemma~\ref{condition}, any two sets within a row intersect in $4$ digits, so the entries within a row also form a colour class. Further, any two entries from different rows and different columns intersect in $3$ digits (as they correspond to $4$-subsets of $abcde$), which violates Condition~1. Hence, such pairs cannot belong to the same colour class. Thus, a colour class with $A$ consists of either entries from a single column or entries from a single row, and hence contains at most six sets from $\mathcal{F}$.
     \end{proof}
  
  \begin{lemma}
   \label{14_4}
       Let $A=abcde$ be a vertex of $3\text{-}QI(11,2)=(V, E)$. Consider a family $\mathcal{F}=\{B\in V: |A \cap B|=0, |B|=4\}$. Then, $|\mathcal{F}|=15$ and at most five from $\mathcal{F}$ along with $A$ can form a colour class in $3\text{-}QI(11, 2).$
  \end{lemma}
  \begin{proof}
        Since the elements of $\mathcal{F}$ are $4$-subsets of $\overline{A}=fghijk$, we have $|\mathcal{F}|=\binom{6}{4}=15$. However, not all of them can belong to the same colour class. By Condition~\ref{adjacencycondition}, for $B_1,B_2\in\mathcal{F}$ to be in the same colour class, $|B_1\cap B_2|\neq 2$. As the subsets are chosen from six elements, the pigeonhole principle implies $|B_1\cap B_2|\ge 2$, and hence $|B_1\cap B_2|=3$. Without loss of generality, let $B_1=fghi$. Then, the next vertex must contain one of the $3$-subsets $fgh$, $ghi$, $fhi$, or $fgi$, and by Lemma~\ref{condition}, any two of these intersect in exactly two digits. The possible options are listed in an array shown in Figure~\ref{14_4table}.
    
    \begin{figure}[h]
        \centering
        \begin{tabular}{|c|c|c|c|}
            \hline
            $fghj$ & $ghij$ & $fhij$& $fgij$\\
            \hline
            $fghk$ & $ghik$& $fhik$&$fgik$\\
            \hline
        \end{tabular}
        \caption{Vertex choices to be in a colour class with $A=abcde$ and $B_1=fghi$ in Lemma~\ref{14_4}.}
        
        \label{14_4table}
    \end{figure}   
    Note that, by Condition~\ref{adjacencycondition}, for each entry, only the entries in the same row or the same column of the array in Figure~\ref{14_4table} can belong to its colour class. Hence, at most $5$ of these $15$ remain valid (see Figure~\ref{tree2}).
   
    \begin{figure}[h]
\centering
\begin{tikzpicture}[
edge from parent/.style={draw, black, thick},
sibling distance=5cm,
level distance=1cm
]
\node (root) {$B_1=fghi$ (WLOG)}
child {node {$B_2=fghj$ (WLOG)}
    child {node {$B_3=fghk$}}
    child {node {$B_3=ghij$ (WLOG)}
        child {node {$B_4=fhij$ (WLOG)}
            child {node {$B_5=fgij$ (WLOG)}}
        }
    }
};
\end{tikzpicture}
\caption{Possibilities of forming colour classes in $\mathcal{F}$, in Lemma \ref{14_4}.}
\label{tree2}
\end{figure}
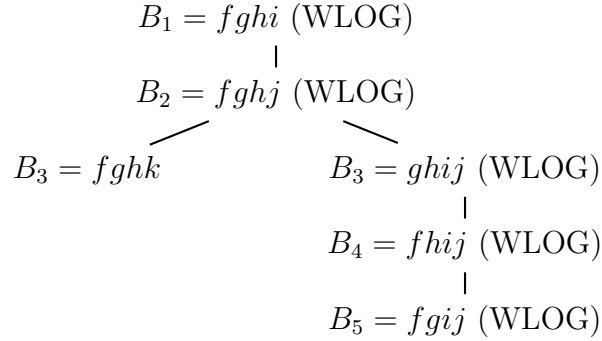

\end{proof}
  
  \begin{lemma}
  \label{14_5}
       Let $A=abcde$ be a vertex of $3\text{-}QI(11,2)=(V, E)$. Consider a family $\mathcal{F}=\{B\in V: |A \cap B|=1, |B|=4\}$. Then at most ten from $\mathcal{F}$ along with $A$ can form a colour class in $3\text{-}QI(11, 2).$
  \end{lemma}

  \begin{proof}
     Each member of $\mathcal{F}$ contains one digit from $A=abcde$ and three from $\overline{A}=fghijk$, so $|\mathcal{F}|=\binom{5}{1}\binom{6}{3}=100$. Let $\mathcal{F}_x=\{B\in V: A\cap B=\{x\},\, |B|=4\}, \quad x=a,b,c,d,e.$ By Condition~\ref{adjacencycondition}, the only possible choices from $\mathcal{F}_x$ that can be in the same colour class (up to permutation) are shown by the three routes in Figure~\ref{tree3}. Hence, each $\mathcal{F}_x$ contributes at most $4$ vertices to a colour class containing $A$, and therefore at most $5\times4=20$ vertices from $\mathcal{F}$ can belong to such a colour class. We next improve this upper bound to $10$. Let $\mathcal{S}\subset \mathcal{F}$ be a set of vertices in a colour class with $A$. \\
    
       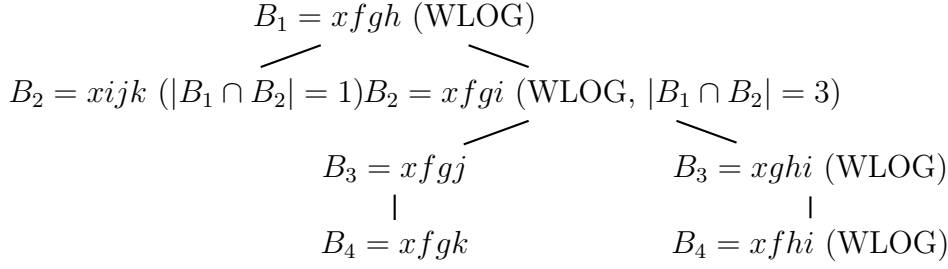
\begin{figure}[h]
       \begin{tikzpicture}[edge from parent/.style={draw, black, thick},sibling distance=5.5cm,level 1/.style={level distance=1cm}]
       \node (root) {$B_1=xfgh$ (WLOG)}
        child {node {$B_2=xijk$ ($|B_1 \cap B_2|=1$)}}
        child {node {$B_2=xfgi$ (WLOG, $|B_1 \cap B_2|=3$)}
        child {node {$B_3=xfgj$}
        child {node {$B_4=xfgk$}}
    }
    child {node {$B_3=xghi$ (WLOG)}
        child {node {$B_4=xfhi$ (WLOG)}}
    }
};

\end{tikzpicture}

\caption{Possibilities of forming colour classes in $\mathcal{F}_x$, in Lemma \ref{14_5}.}
\label{tree3}
\end{figure}

\noindent \textbf{Case 1.} Suppose $\mathcal{S}$ contains two vertices $B_1 \in \mathcal{F}_x$ and $B_2 \in \mathcal{F}_y$, where $x \ne y$ and $B_1 \cap B_2 = \emptyset$. Without loss of generality, let $B_1 = xfgh$ and $B_2 = yijk$. Then, for each $z \notin \{x,y\}$, the family $\mathcal{F}_z$ can contribute at most two vertices to $\mathcal{S}$, namely, $zfgh$ and $zijk$. Moreover, besides $B_1$, the family $\mathcal{F}_x$ can contribute at most three additional vertices from:
$xijk$ (intersecting $B_1$ in $1$ digit), $xfgi,\, xfgj,\, xfgk,\, xfhi,\, xfhj,\, xfhk,\, xghi,\, xghj,\,$ $ xghk$ (each intersecting $B_1$ in $3$ digits). Similarly, in addition to $B_2$, $\mathcal{F}_y$ can contribute at most three vertices from: $yfgh$ (intersecting $B_2$ in $1$ digit), $yfij,\, ygij,\, yhij,\, yfik,\, ygik,\,$ $ yhik,\, yfjk,\, ygjk,\, yhjk$ (each intersecting $B_2$ in $3$ digits). Now, the following cases arise for the choice of $B_3$. 
\begin{enumerate}
    \item[(i)] If $B_3=xijk\in \mathcal{S}$, then no further vertex can be chosen from $\mathcal{F}_x$, and the only possible vertex from $\mathcal{F}_y$ is $yfgh$. Thus, each family $\mathcal{F}_x$, $x\in\{a, b, c, d, e\}$ contributes at most two vertices and hence, $|\mathcal{S}|\leq 10$.
    
    \item[(ii)] If $xijk\notin \mathcal{S}$, assume without loss of generality that $B_3=xfgi \in \mathcal{S}$. Then the remaining options in $\mathcal{F}_y$ are $yhij, yhik,$ and $yhjk$ (each intersecting $B_3$ in $1$ digit), while for each $\mathcal{F}_z$, $z \notin {x,y}$, only $zijk$ can be chosen (also intersecting $B_3$ in one digit). If $\mathcal{S}\subseteq \mathcal{F}_x\cup \mathcal{F}_y$, then $|S|\leq 8$. Otherwise, if $zijk\in \mathcal{S}$ for some $z \notin {x,y}$, then $yhij, yhik,$ and $yhjk$ become unavailable, yielding at most $4\,(\mathcal{F}_x)+1\, (B_2)+ 1\,(\mathcal{F}_z)+1\,(\mathcal{F}_u)+1\,(\mathcal{F}_v)=8$ vertices in $\mathcal{S}$. 
    \item[(iii)] If $\mathcal{S}\cap \mathcal{F}_x=\{B_1\}$, then two cases arise. If $|\mathcal{F}_y\cap \mathcal{S}|\leq 2$, that is, $\mathcal{F}_y\cap \mathcal{S}\subseteq\{yijk, yfgh\}$, then $|\mathcal{S}|\leq 1\, (B_1)+ 2\, (\mathcal{F}_y)+ 2\, (\mathcal{F}_z) + 2\, (\mathcal{F}_u) + 2\,(\mathcal{F}_v)=9$ while if $|\mathcal{F}_y\cap \mathcal{S}|>2 $, that is, when $\mathcal{S}$ has at least one vertex from $\mathcal{F}_z$ that intersects $B_2$ in $3$ digits, then $\mathcal{F}_z$ can contribute only $zfgh$, and hence $|\mathcal{S}|\leq1\, (B_1)+ 4\, (\mathcal{F}_y)+ 1\, (\mathcal{F}_z) + 1\, (\mathcal{F}_u) + 1(\mathcal{F}_v)=8$. 
\end{enumerate}

\noindent \textbf{Case 2.}  
In $\mathcal{S}$, every pair of vertices $B_1$ and $B_2$ from different families $\mathcal{F}_x$, $x \in {a,b,c,d,e}$, intersect in exactly $1$ or $3$ digits. \\

\noindent \textit{Case 2.1.} If $|B_1\cap B_2|=1$, assume without loss of generality that $B_1=xfgh$ and $B_2=yfij$. Then, besides $B_1$, the family $\mathcal{F}_x$ can contribute at most two additional vertices from: $xfgk, xfhk, xghi, xghj$ (each intersecting $B_1$ in $3$ digits and $B_2$ in $1$ digit). Similarly, besides $B_2$, $\mathcal{F}_y$ can contribute at most two vertices from: $yfik, yfjk, ygij, yhij$ (each intersecting $B_1$ in $1$ digit and $B_2$ in $3$ digits). Moreover, each $\mathcal{F}_z$, $z\notin \{x, y\}$, can contribute either $zfgh$ (intersecting $B_1$ in 3 and $B_2$ in $1$ digit), or  $zfij$ (intersecting $B_1$ in $1$ and $B_2$ in $3$ digits), or at most two vertices from: $zgik, zgjk, zhik, zhjk$ (intersecting $B_1$ in $1$ and $B_2$ in $1$ digit).  We examine these possibilities for $B_3\in \mathcal{F}_z$ in turn.
\begin{enumerate}
   \item[(i)] If $|\mathcal{F}_z\cap \mathcal{S}|\leq 1$ for every $z\notin \{x, y\}$, then $|\mathcal{S}|\leq 3\,(\mathcal{F}_x) + 3\,(\mathcal{F}_y) + 1\,(\mathcal{F}_z) + 1\,(\mathcal{F}_u) + 1\,(\mathcal{F}_v)= 9$.   
  \item[(ii)] If $|\mathcal{F}_z\cap \mathcal{S}|= 2$, say $zgik$ and $zgjk$ (WLOG), then each of $\mathcal{F}_u$ and $\mathcal{F}_v$ can contribute at most one vertex $(ufgh \text{ or } ufij; vfgh \text{ or } vfij)$. Hence $|\mathcal{S}|\leq 3\,(\mathcal{F}_x) + 3\,(\mathcal{F}_y) + 2\,(\mathcal{F}_z)+1\,(\mathcal{F}_u) + 1\,(\mathcal{F}_v)= 10$.\\
\end{enumerate}

\noindent \textit{Case 2.2.} If $|B_1\cap B_2|=3$, assume without loss of generality that $B_1=xfgh$ and $B_2=yfgh$. By Condition~\ref{adjacencycondition}, each $\mathcal{F}_z$, $z\notin\{x, y\}$, can contribute at most three vertices from: $zfgh$ (intersecting $B_1$ and $B_2$ in $3$ digits), $zfij, zfik, zfjk, zgij, zgik, zgjk, zhij, zhik,$ $ zhjk$ (intersecting $B_1$ and $B_2$ in $1$ digit). 
\begin{enumerate}
    \item[(i)] If $B_3=zfgh$, then $\mathcal{F}_z$ cannot contribute any additional vertex, since any such vertex would intersect $B_3$ in $2$ digits. Moreover, any vertex from $\mathcal{F}_x$ intersecting $B_1$ at $3$ digits would intersect $B_2$ at $2$ digits and is therefore invalid. Hence, besides $B_1$, the only possible vertex from $\mathcal{F}_x$ is $xijk$ (intersecting $B_1$ in $1$ digit). If $B_4=xijk$, then $B_3, B_4\in\mathcal{S}$ and the result follows from Case~1. Similarly, if $B_4=yijk$ then $B_3, B_4\in\mathcal{S}$ and the result follows from Case~1. Otherwise, if $B_4\notin \mathcal{F}_x\cup\mathcal{F}_y$, then $|\mathcal{S}|\leq 1\, (B_1)+ 1\, (B_2)+ 1\, (B_3) + 3\, (\mathcal{F}_u) +3\, (\mathcal{F}_v) =9$. 
    \item[(ii)] If $zfgh\notin \mathcal{S}$, assume without loss of generality that $B_3=zfij\in \mathcal{S}$. Then $|B_1\cap B_3|=1$ and result follows from Case~2.1. 
\end{enumerate}

Therefore, in all the cases, $|\mathcal{S}|\leq 10$, which establishes the claim. \end{proof}
 
  \begin{lemma}
  \label{14_6}
       Let $A=abcde$ be a vertex of $3\text{-}QI(11,2)=(V, E)$. Consider a family $\mathcal{F}=\{B\in V: |A \cap B|=3, |B|=4\}$. Then, at most eight from $\mathcal{F}$ along with $A$ can form a colour class in $3\text{-}QI(11, 2).$
  \end{lemma}
  
  \begin{proof}
  With a fixed $3$-subset of $A$, there are $6$ vertices in $\mathcal{F}$, corresponding to each digit in $\overline{A}$ and hence $|\mathcal{F}|=\binom{5}{3} \binom{6}{1} = 60$. Let $\mathcal{F}_x = \{B\in \mathcal{F}: x \in \overline{A}\cap B, |B|=4\}$, for $x\in\{f, g, h, i, j, k\}$. By Theorem \ref{AK ekr} (taking $n=5, k=3, t=2$), each $\mathcal{F}_x$ can contribute at most $4$ vertices (pairwise intersecting in $3$ digits) in a colour class containing $A$ as shown in Figure \ref{tree5}. 
  
  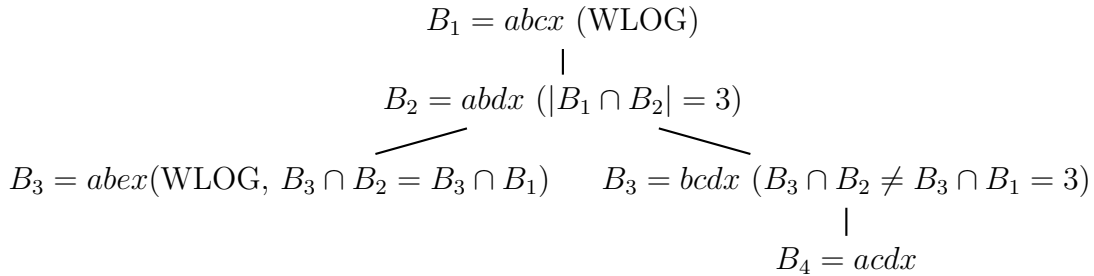
\begin{figure}[h]   
  \begin{tikzpicture}[edge from parent/.style={draw, black, thick}, sibling distance=5 cm, level distance=0.7cm, scale=1.5]
      \node (root) {$B_1=abcx$ (WLOG)}
      child{node {$B_2=abdx$ ($|B_1 \cap B_2|=3$)} child{node {$B_3=abex$(WLOG, $B_3 \cap B_2=B_3 \cap B_1$)}}child {node {$B_3=bcdx$ ($B_3 \cap B_2 \neq B_3 \cap B_1=3$)}child {node{$B_4=acdx$}}} };\end{tikzpicture}
      \caption{Possibilities for forming colour classes in $\mathcal{F}_x$ in Lemma \ref{14_6}.}
      \label{tree5}
     \end{figure}
     Let $\mathcal{S}\subset \mathcal{F}$ be a set of vertices in a colour class with $A$. We will prove that $|\mathcal{S}|\leq 8$. \\

     Without loss of generality, let $B_1=abcf\in \mathcal{S}$. Then, besides $B_1$, $\mathcal{F}_f$ can contribute  at most three vertices from: $
      abdf, abef, bcdf, bcef, acdf, acef$ (each intersecting $B_1$ in $3$ digits). Moreover, each $\mathcal{F}_x$, $x\neq f$, can contribute either $abcx$, or at most $3$ vertices: $adex, bdex, cdex$. Now, we examine these possibilities for $B_2\in \mathcal{F}_x$, $x\neq f$.\\ 
      
      \noindent\textbf{Case 1.} 
      If every $\mathcal{F}_x$, $x\in\{g, h, i, j, k\}$, contributes only $abcx$,  then $\mathcal{F}_f$ cannot contribute any additional vertex, since each would intersect $abcx$ in $2$ digits. Hence, $|\mathcal{S}|\leq 1+1+1+1+1+1=6$.  \\
      
      \noindent\textbf{Case 2.} If at least one of $\mathcal{F}_x$, $x\in\{g, h, i, j, k\}$, contributes a vertex $B_2\in\{adex, bdex,$ $ cdex\}$. Without loss of generality, let $B_2=adeg$ from $\mathcal{F}_g$, then each of the remaining  $\mathcal{F}_x$, $x\in\{h, i, j, k\}$ can contribute at most one vertex from $abcx$ and $adex$. Moreover, it leaves only two additional vertices available in  $\mathcal{F}_f$, namely $bcdf$ and $bcef$. Hence, if $\mathcal{F}_g$ contributes only $B_2=adeg$ then $|\mathcal{S}|\leq 3+ 1+ 1 + 1+ 1+ 1=8$. Otherwise, if $\mathcal{F}_g$ contributes more than one vertex to $\mathcal{S}$, then $\mathcal{F}_f$ can contribute only $B_1=abcf$ and $|\mathcal{S}|\leq 1+3+ 1+1 +1 +1=8$. 
  \end{proof}
  
  \begin{lemma}
  \label{14_7} 
      Let $A=abcde$ be a vertex of $3\text{-}QI(11,2)=(V, E)$. Consider a family $\mathcal{F}=\{B\in V: |A \cap B|=4, |B|=4\}$. Then, $|\mathcal{F}|=5$ and these five along with $A$ can form a colour class in $3\text{-}QI(11, 2).$
  \end{lemma}
  \begin{proof}
      $\mathcal{F}=\{abcd,
	abce,
	abde,
	acde,
	bcde\}$. By Lemma \ref{condition}, any two sets intersect in three digits because each $4$-subset is chosen from the $5$-element set $abcde$. Hence, by Condition \ref{adjacencycondition}, all sets in $\mathcal{F}$ together with $A$ form a valid colour class. 
  \end{proof}
   
   \begin{lemma}
      \label{14_8}
       Let $A=abcd$ be a vertex of $3\text{-}QI(11,2)=(V, E)$. Consider a family $\mathcal{F}=\{B\in V: |A \cap B|=0, |B|=4\}$. Then, at most five from $\mathcal{F}$ along with $A$ can form a colour class in $3\text{-}QI(11, 2).$
  \end{lemma}
 
  \begin{proof}
  The vertices in $\mathcal{F}$ are $4$-subsets of the $7$-set $\overline{A}=efghijk$, so $|\mathcal{F}|=\binom{7}{4}=35$. Without loss of generality, let $B_1=efgh$. We consider vertices in $\mathcal{F}$ that can belong to a colour class $\mathcal{S}$ containing $A$ and $B_1$. Each such vertex must intersect $B_1$ in either $1$ or $3$ digits. These vertices are arranged in Figure \ref{Table14_8}, where the first three columns correspond to intersection size $3$ with $B_1$, and the last column corresponds to intersection size $1$.
  \begin{figure}[h]
  \centering
          \begin{tabular}{|c|c|c|c|}
          \hline
              $efgi$ & $efgj$ & $efgk$ & $hijk$ \\
              \hline
               $eghi$ & $eghj$ & $eghk$ & $fijk$\\
               \hline
               $fghi$ & $fghj$ & $fghk$ & $eijk$\\
               \hline
               $efhi$ & $ efhj$ & $efhk$ & $gijk$\\
               \hline
          \end{tabular}
          \caption{Vertex choices to be in a colour class with $A=abcd$ and $B_1=efgh$ in Lemma \ref{14_8}.}
          \label{Table14_8}
      \end{figure}
     By Lemma~\ref{condition}, any two vertices within a column of Figure~\ref{Table14_8}, intersect in $3$ digits. Similarly, in each row, the last vertex intersects the others in $1$ digit, while the remaining vertices intersect pairwise in $3$ digits. Hence, vertices within the same column or the same row can belong to $\mathcal{S}$. However, any pair of vertices from different rows and columns intersect in $2$ digits and thus cannot be in the same colour class. Therefore, Figure~\ref{Table14_8} contributes at most four vertices, and hence, $|\mathcal{S}|\leq 1+4=5$.
  \end{proof}
  
  \begin{lemma}
      Let $A=abcd$ be a vertex of $3\text{-}QI(11,2)=(V, E)$. Consider a family $\mathcal{F}_x=\{B\in V: A \cap B=\{x\}, |B|=4\}$ for $x\in{a,b,c,d}$. Then, at most five from $\mathcal{F}_x$ along with $A$ can form a colour class in $3\text{-}QI(11, 2).$
      \label{13_familywise}
  \end{lemma} 

  \begin{proof} Let $B_1 = aefg \in \mathcal{F}_a$. We determine the maximum number of vertices in $\mathcal{F}_a$ that can belong to the same colour class as $B_1$. By Condition~\ref{adjacencycondition}, such vertices must intersect $efg$ in either $0$ or $2$ digits. The possibilities are summarized in the array shown in Figure~\ref{Table12_new}, where each of the first three columns corresponds to a choice of a $2$-subset of $efg$, while each row (excluding the last column) corresponds to an element of $\overline{A \cup B_1}$. The final column lists the $3$-subsets disjoint from $efg$, obtained as complements of the corresponding entries in that row.
  \begin{figure}[h] 
  \centering
   \begin{tabular}{|c|c|c|c|}
  \hline
      $aefh$ & $afgh$ & $aegh$ & $aijk$  \\
      \hline
       $aefi$ & $afgi$ & $aegi$ & $ahjk$\\
       \hline
       $aefj$ & $afgj$ & $aegj$ & $ahik$  \\
       \hline
       $aefk$ & $afgk$ & $aegk$ & $ahij$\\
       \hline       
  \end{tabular}
  \caption{Vertex choices to be in a colour class with $A=abcd$ and $B_1=aefg$ in Lemma~\ref{13_familywise}.}
  \label{Table12_new}
  \end{figure}
  By Lemma~\ref{condition} and Condition~\ref{adjacencycondition}, vertices within the same row or column can form a colour class, whereas any pair of vertices from different rows and columns cannot be in the same colour class. Hence, at most five vertices from $\mathcal{F}_x$, together with $A$, can form a colour class in $3\text{-}QI(11,2)$.\end{proof}

    \begin{lemma}
      \label{12_2}
       Let $A = abcd$ be a vertex of $3\text{-}QI(11,2) = (V,E)$, and let $\mathcal{F} = \{ B \in V : |A \cap B| = 1,\ |B| = 4 \}$. If a colour class in $3\text{-}QI(11,2)$ containing $A$ includes two vertices $B_1, B_2 \in \mathcal{F}$ such that $|B_1 \cap B_2| = 1$ and $A \cap B_1 = A \cap B_2$, then it can have at most $12$ vertices from $\mathcal{F}$ (including  $B_1$ and $B_2$).
      
  \end{lemma}
  \begin{proof}
      Without loss of generality, let $B_1=aefg$ and $B_2=ahij$ be in a colour class $\mathcal{S}$ containing $A$. Let $\mathcal{F}_x=\{B\in \mathcal{F} : A\cap B=\{x\}\}$.  Then by Lemma \ref{13_familywise}, besides $B_1$ and $B_2$, $\mathcal{F}_a$ can contribute at most $3$ vertices in $\mathcal{S}$. While for each $x\in\{b, c, d\}$, $\mathcal{F}_x$ contributes, either at most $2$ from $xefg$ (intersecting $B_1$ in $3$ and $B_2$ in $1$ digit) and $xhij$ (intersecting $B_1$ in $1$ and $B_2$ in $3$ digits), or at most $3$ from:  
      \[xehk, xeik, xejk, xfhk, xfik, xfjk, xghk, xgik, xgjk\]
      (each intersecting $B_1$ and $B_2$ in $1$ digit). \\
      
      \noindent \textbf{Case 1.} If none of the $\mathcal{F}_x, x\in\{b, c, d\}$ contributes a vertex that intersects $B_1$ and $B_2$ in $1$ digit each, then each $\mathcal{F}_x$ contributes at most $2$ vertices. Hence, $|\mathcal{S}|\leq 5\, (\mathcal{F}_a)+ 2\, (\mathcal{F}_b)+ 2\, (\mathcal{F}_c)+ 2\, (\mathcal{F}_d) =11$. \\
      
      \noindent \textbf{Case 2.} If at least one $\mathcal{F}_x$ contributes a vertex $B_3$ that intersects $B_1$ and $B_2$ in $1$ digit each. Assume without loss of generality that $x=b$ and $B_3=behk$. Then, besides $B_3$,  $\mathcal{F}_b$ can now contribute at most $2$ vertices from: $beik, bejk, bfhk, bghk$ (each intersecting $B_3$ in $3$ digits). Moreover, this also reduces the choices in each $\mathcal{F}_y$, $y\neq b$, and now each $\mathcal{F}_y$ can contribute at most $2$ from either $yefg$ and $yhij$, or from $yehk$ (intersecting $B_3$ in $3$ digits), $yfik, yfjk, ygik, ygjk$ (each intersecting $B_3$ in $1$ digit). Therefore, $|\mathcal{S}|\leq 5\,(\mathcal{F}_a) + 3\, (\mathcal{F}_b) + 2\, (\mathcal{F}_c)+ 2\, (\mathcal{F}_d)=12$.     
  \end{proof}
  
\begin{lemma}
      \label{12_1}
        Let $A = abcd$ be a vertex of $3\text{-}QI(11,2) = (V,E)$, and let $\mathcal{F} = \{ B \in V : |A \cap B| = 1,\ |B| = 4 \}$. If a colour class in $3\text{-}QI(11,2)$ containing $A$ includes two vertices $B_1, B_2 \in \mathcal{F}$ such that $|B_1 \cap B_2| = 3$ and $A \cap B_1 = A \cap B_2$, then it can have at most $13$ vertices from $\mathcal{F}$ (including  $B_1$ and $B_2$).
  \end{lemma}  
\begin{proof}

Without loss of generality, let $B_1=aefg$ and $B_2=aefh$ be vertices in a colour class $\mathcal{S}$ containing $A$. Let $\mathcal{F}_x=\{B\in \mathcal{F} : A\cap B=\{x\}\}, x \in \{a,b,c,d\}$. Then by Lemma \ref{13_familywise}, besides $B_1$ and $B_2$, $\mathcal{F}_a$ can contribute at most $3$ vertices in $\mathcal{S}$. 
If $\mathcal{S}$ contains two vertices from some $\mathcal{F}_x$ whose intersection is exactly one digit (namely $x$), then the result follows from Lemma~\ref{12_2}. Therefore, for the remainder of the proof, we assume that no $\mathcal{F}_x$ contributes such a pair of vertices. Consequently, every pair of vertices in $\mathcal{F}_x \cap \mathcal{S}, x\in \{a, b, c, d\}$, intersects in exactly three digits. With the given choice of $B_1$ and $B_2$, the family $\mathcal{F}_a$ can contribute at most three additional vertices to $\mathcal{S}$, chosen either from $aefi, aefj, aefk$ or from $aegh, afgh$. Based on whether there exists some $\mathcal{F}_x$, $x \neq a$, that contributes a vertex $xijk$ (disjoint from both $B_1$ and $B_2$) to $\mathcal{S}$, we consider the following two cases. \\

\noindent \textbf{Case 1.}  If for some $x\neq a$, $\mathcal{F}_x$ contributes the vertex $xijk$, then we assume without loss of generality that $x=b$ and $B_3=bijk\in \mathcal{S}$. Figure \ref{fig13case} shows the remaining valid options in $\mathcal{F}_b$, each intersecting $B_3$ in $3$ digits.
\begin{figure}[h]
\centering
    \begin{tabular}{|c|c|c|c|} \hline $beik$ & $bfik$ & $bgik$ & $bhik$ \\ \hline $bejk$ & $bfjk$ & $bgjk$ & $bhjk$ \\ \hline $beij$ & $bfij$ & $bgij$ & $bhij$ \\ \hline \end{tabular}
    \caption{Vertex choices to be in a colour class with $A=abcd$, $B_1=aefg$, $B_2= aefh$ and $B_3=bijk$ in Lemma~\ref{12_1}.}
    \label{fig13case}
    \end{figure}
With the given choices of $A, B_1, B_2,$ and $B_3$, each $\mathcal{F}_x$, for $x \in \{c,d\}$, can contribute either the vertex $xijk$ or at most three vertices from $xghi, xghj, xghk$. If at least one of $cijk$ and $dijk$ is included in $\mathcal{S}$, then by Condition~\ref{adjacencycondition}, $\mathcal{S}$ cannot contain any vertices from Figure~\ref{fig13case}. Consequently, $|\mathcal{S}|\leq 5\,(\mathcal{F}_a) + 1\, (\mathcal{F}_b) + 4\, (\mathcal{F}_c \cup \mathcal{F}_d)=10$. On the other hand, if neither $cijk$ nor $dijk$ is included, then $\mathcal{F}_c \cup \mathcal{F}_d$ can contribute at most three vertices - either all from the same $\mathcal{F}_x$, or at most two vertices, one from each of $\mathcal{F}_c$ and $\mathcal{F}_d$. Hence, $|\mathcal{S}|\leq 5\,(\mathcal{F}_a) + 5\, (\mathcal{F}_b) + 3\, (\mathcal{F}_c \cup \mathcal{F}_d)=13$.\\

\noindent \textbf{Case 2.} If, for all $x$, no $\mathcal{F}_x$ contributes the vertex $xijk$, then $\mathcal{F}_b$ contributes at most four vertices from Figure~\ref{fig13case2}. 
\begin{figure}[h]
        \centering
        \begin{tabular}{|c|c|c|c||c|}
        \hline
            $beij$ & $bfij$ & $bgij$ & $bhij$ & $bghi$ \\
            \hline
            $beik$ & $bfik$ & $bgik$ & $bhik$ & $bghj$\\
            \hline
            $bejk$ & $bfjk$ & $bgjk$ & $bhjk$ & $bghk$\\
            \hline
        \end{tabular}
        \caption{Vertex choices to be in a colour class with $A=abcd$, $B_1=aefg$ and $B_2= aefh$ in Lemma~\ref{12_1}.}
        \label{fig13case2}
\end{figure}
    By Condition~\ref{adjacencycondition}, vertices within a column of the array in Figure~\ref{fig13case2} can belong to the same colour class. However, in the $3\times 4$ subarray formed by the first four columns, no two vertices from different rows and columns can be in the same colour class. Moreover, at most two vertices from any row or column of this subarray can be included in a colour class together with a vertex from the fifth column. Hence, $\mathcal{F}_b$ contributes at most $4$ vertices in $\mathcal{S}$, each of which intersects $B_1$ and $B_2$ in $0, 1,$ or $3$ digits. The only vertex that is disjoint from both $B_1$ and $B_2$ is $bijk$ and it has been discussed in Case 1. Moreover, there does not exist any vertex in $\mathcal{F}_b$ that intersects both $B_1$ and $B_2$ in $3$ digits. Therefore, based on the nature of intersection with $B_1$ and $B_2$, this leads to the following three cases for a vertex $B_3\in \mathcal{F}_b\cap \mathcal{S}$. 
    \begin{enumerate}
        \item[(i)] If $|B_3\cap B_1|=0$ and $|B_3\cap B_2|=1$, then without loss of generality, we assume $B_3=bhij$. Consequently, each $\mathcal{F}_x$, $x\in\{c, d\}$ can contribute either $xhij$ or $xghk$, or at most three vertices chosen from the same row or the same column, as shown below. 
        \begin{center}
        \begin{tabular}{ccc}
         $xfik$ & $xgik$ & $xhik$ \\
         $xfjk$ & $xgjk$ & $xhjk$
        \end{tabular}
        \end{center}
    If $xhij$ or $xghk$ is selected, then $|\mathcal{S}|\leq 5\,(\mathcal{F}_a) + 4\, (\mathcal{F}_b) + 4\, (\mathcal{F}_c \cup \mathcal{F}_d)=13$. Any other choice from $\mathcal{F}_x$, reduces the number of choices from $\mathcal{F}_y$, $y\notin \{a,b,x\}$ by at least $1$, in which case $|\mathcal{S}|\leq 5\,(\mathcal{F}_a) + 4\, (\mathcal{F}_b) + 3\, (\mathcal{F}_c \cup \mathcal{F}_d)=12$.
        \item[(ii)] If $|B_3\cap B_1|=|B_3\cap B_2|=1$ and $B_3\cap B_1=B_3\cap B_2$, then without loss of generality, let $B_3=beij$. In this case, each $\mathcal{F}_x$, $x\in\{c, d\}$ can contribute either $xeij$ or $xghi$, or at most three vertices chosen from the same row or the same column, as shown below. 
        \begin{center}
        \begin{tabular}{ccc}
         $xfik$ & $xgik$ & $xhik$ \\
         $xfjk$ & $xgjk$ & $xhjk$
        \end{tabular}
        \end{center}
    If $xeij$ or $xghi$ is selected, then $|\mathcal{S}|\leq 5\,(\mathcal{F}_a) + 4\, (\mathcal{F}_b) + 4\, (\mathcal{F}_c \cup \mathcal{F}_d)=13$. Any other choice from $\mathcal{F}_x$, reduces the number of choices from $\mathcal{F}_y$, $y\notin \{a,b,x\}$ by at least $1$, yielding $|\mathcal{S}|\leq 5\,(\mathcal{F}_a) + 4\, (\mathcal{F}_b) + 3\, (\mathcal{F}_c \cup \mathcal{F}_d)=12$. 
        \item[(iii)] If $|B_3\cap B_1|=|B_3\cap B_2|=1$ and $B_3\cap B_1\neq B_3\cap B_2$, then without loss of generality, let $B_3=bghi$. Consequently, each $\mathcal{F}_x$, $x\in\{c, d\}$ can contribute at most two vertices from the following. 
        \[xghi, xfjk, xfik, xgjk, xhjk\]
    Therefore, $|\mathcal{S}|\leq 5\,(\mathcal{F}_a) + 4\, (\mathcal{F}_b) + 2\, (\mathcal{F}_c) + 2\, (\mathcal{F}_d)=13$.       
    \end{enumerate}

\end{proof}    
 
  \begin{proposition}
      \label{14_9}
       Let $A=abcd$ be a vertex of $3\text{-}QI(11,2)=(V, E)$. Consider a family $\mathcal{F}=\{B\in V: |A \cap B|=1, |B|=4\}$. Then at most thirteen from $\mathcal{F}$ along with $A$ can form a colour class in $3\text{-}QI(11, 2).$
  \end{proposition}
 
  \begin{proof}
      Each vertex in $\mathcal{F}$ contains one digit from $A=abcd$ and three digits from $\overline{A}=efghijk$, so $|\mathcal{F}|=\binom{4}{1}\binom{7}{3}=140$. Let $\mathcal{F}_x=\{B\in V: A\cap B=\{x\}\}$ for $x\in{a,b,c,d}$. If $\mathcal{S}\subseteq \mathcal{F}$ is a colour class with $A=abcd$ and each $\mathcal{F}_x$ contributes at most one vertex to $\mathcal{S}$, then $|\mathcal{S}|\le 4$, proving the result.\\
      
      If at least one $\mathcal{F}_x$ contributes more than one vertex to $\mathcal{S}$, assume without loss of generality that $x=a$ and let $B_1=aefg\in\mathcal{S}$. By Condition~\ref{adjacencycondition}, apart from $B_1$, $\mathcal{S}$ can include only those vertices of $\mathcal{F}_a$ that intersect $B_1$ in either $1$ digit ($a$) or $3$ digits (two digits from $efg$). 
      This leads to the following two cases.\\
      
      \noindent \textbf{Case 1.} If $B_2\in \mathcal{S}\cap \mathcal{F}_a$ and $|B_1\cap B_2|=1$, assume without loss of generality that $B_2=ahij$. Since $A\cap B_1=A\cap B_2={a}$, by Lemma~\ref{12_2}, we obtain $|\mathcal{S}|\le 12$.\\
      \noindent \textbf{Case 2.} Let $B_2\in \mathcal{S}\cap \mathcal{F}_a$ with $|B_1\cap B_2|=3$. Without loss of generality, take $B_2=aefh$. Since $A\cap B_1=A\cap B_2={a}$, by Lemma~\ref{12_1}, we obtain $|\mathcal{S}|\le 13$.    
  \end{proof}

  \begin{lemma}
      \label{14_10}
       Let $A=abcd$ be a vertex of $3\text{-}QI(11,2)=(V, E)$. Consider a family $\mathcal{F}=\{B\in V: |A \cap B|=3, |B|=4\}$. Then at most seven from $\mathcal{F}$ along with $A$ can form a colour class in $3\text{-}QI(11, 2).$
  \end{lemma}
  
  \begin{proof}
       Since every vertex in $\mathcal{F}$ has three digits from $A=abcd$ and one from $\overline{A}=efghijk$, $|\mathcal{F}|=\binom{4}{3}.\binom{7}{1}=28$ and these vertices are shown in Figure~\ref{14_10}, where rows correspond to $3$-subsets of $A=abcd$. 
       \begin{figure}[h]
\centering
\begin{tabular}{|c|c|c|c|c|c|c|}
\hline
$abce$ & $abcf$ & $abcg$ & $abch$ & $abci$ & $abcj$ & $abck$\\
\hline
$abdf$ & $abdg$ & $abdh$ & $abdi$ & $abdj$ & $abdk$ & $abde$\\
\hline
$acdg$ & $acdh$ & $acdi$ & $acdj$ & $acdk$ & $acde$ & $acdf$\\
\hline
$bcdh$ & $bcdi$ & $bcdj$ & $bcdk$ & $bcde$ & $bcdf$ & $bcdg$\\
\hline
\end{tabular}
\caption{The family $\mathcal{F} = \{B \in V: |A \cap B| = 3, |B|=4\}$ in Lemma \ref{14_10}.}
\label{14_101}
\end{figure}
       Since vertices within a row in Figure~\ref{14_101} intersect in $3$ digits, they can form a colour class with $A$. By Lemma~\ref{adjacencycondition}, two vertices from different rows share the same colour only if they contain the same digit from $\overline{A}=efghijk$. Thus, choosing a vertex in one row (e.g., $abce$) invalidates all the vertices in the other rows except one in each (vertices containing $e$). Consequently, selecting more than one vertex in a row eliminates all vertices from the remaining rows. If one vertex from each row is selected, then $\mathcal{F}$ contributes only four vertices.  Hence, at most $7$ vertices from $\mathcal{F}$ can belong to a colour class with $A$, all from the same row.
  \end{proof}
  
  \section{A Lower Bound on $\chi_S(3\text{-}QI(11,2))$ and its Implication for $CAN(3\text{-}QI(11,2),2)$}
  \label{sec4}
  
  The upper bounds on the number of vertices in a colour class from various families $\mathcal{F}$, obtained in the previous section, are used to determine the strong independence number $\alpha_S(3\text{-}QI(11,2))$, and hence yield a lower bound on the strong chromatic number of $3\text{-}QI(11,2)$, which is then used to determine its covering array number.\\

  To determine $\alpha_S(3\text{-}QI(11,2))$, we first show that the maximum size of a colour class in $3\text{-}QI(11,2)$ is 26. Suppose a colour class contains no vertex corresponding to a $5$-subset of $[11]$, and let $A=abcd$ be any vertex in it. Then any other vertex $B$ (a $4$-subset) in the same colour class must, by Condition~\ref{adjacencycondition}, satisfy $|A\cap B|\in \{0, 1, 3\}$. Accordingly, the number of such vertices is at most $5, 13,$ and $7$, respectively, by Lemma~\ref{14_8}, Proposition~\ref{14_9}, and Lemma~\ref{14_10}. Thus, the number of vertices in such a colour class is at most $1 (A=abcd)+5+13+7=26$. \\

Next, we determine the maximum size of a colour class containing at least one vertex corresponding to a $5$-subset of $[11]$. Without loss of generality, let $A=abcde$ be such a vertex. Then, by Condition~\ref{adjacencycondition},  any other vertex $B$ in the same colour class satisfies the following.

\begin{enumerate}
\item[\textbf{Case 1.}] $B$ corresponds to a $5$-subset.
\begin{enumerate}
    \item[(a)] $|A\cap B|=0$. By Lemma~\ref{14_1}, at most $6$ such vertices can be included.
    \item[(b)] $|A\cap B|=1$. By Lemma~\ref{14_2}, at most $5$ such vertices can be included.
    \item[(c)] $|A\cap B|=4$. By Lemma~\ref{14_3}, at most $6$ such vertices can be included.
\end{enumerate}
\item[\textbf{Case 2.}] $B$ corresponds to a $4$-subset. 
\begin{enumerate}
    \item[(a)] $|A\cap B|=0$. By Lemma~\ref{14_4}, at most $5$ such vertices can be included.
    \item[(b)] $|A\cap B|=1$. By Lemma~\ref{14_5}, at most $10$ such vertices can be included.
    \item[(c)] $|A\cap B|=3$. By Lemma~\ref{14_6}, at most $8$ such vertices can be included.
    \item[(d)] $|A\cap B|=4$. By Lemma~\ref{14_7}, at most $5$ such vertices can be included.
\end{enumerate}
\end{enumerate}
    Thus, an initial upper bound on the size of a colour class containing a $5$-subset is $1 (A=abcde)+6+5+6+5+10+8+5=46$. However, these bounds cannot be attained simultaneously. In what follows, we refine this estimate by analysing the mutual constraints arising from the interaction of the above cases. \\

\begin{lemma}
Let $A=abcde$ be a vertex of $3\text{-}QI(11,2)$, and define
\[
\mathcal{F}_1=\{B : |A\cap B|=0,\ |B|=5\}, \quad
\mathcal{F}_2=\{B : |A\cap B|=1,\ |B|=5\}.
\]
Then $|\mathcal{F}_1 \cup \mathcal{F}_2| \le 7$ within any colour class containing $A$. Moreover, equality holds if and only if the seven vertices share a common $4$-subset of $\overline{A}=fghijk$.
\label{comb1bc}
\end{lemma}

 \begin{proof} 
 Individually, the maximum contributions of $\mathcal{F}_1$ and $\mathcal{F}_2$ to a colour class are $6$ and $5$, respectively, by Cases~1(a) and~1(b) in Section~\ref{sec4}. Vertices in $\mathcal{F}_2$ contain exactly four digits from $\overline{A}=fghijk$, while vertices in $\mathcal{F}_1$ are precisely the $5$-subsets of $\overline{A}$. Hence, by the pigeonhole principle, any vertex in $\mathcal{F}_2$ intersects each vertex in $\mathcal{F}_1$ in either three or four digits. Consequently, the inclusion of any vertex from $\mathcal{F}_2$ restricts the admissible vertices from $\mathcal{F}_1$ to at most two, namely those intersecting it in four digits. Thus, if at least one vertex from $\mathcal{F}_2$ is chosen, then at most $2+5=7$ vertices can be included in total. Equality is achieved precisely when the selected vertices share a common $4$-subset of $\overline{A}$; for example,
\[
fghij,\; fghik \in \mathcal{F}_1, \quad
afghi,\; bfghi,\; cfghi,\; dfghi,\; efghi \in \mathcal{F}_2.
\]
\end{proof}


     \begin{lemma} 
     Let $A=abcde$ be a vertex of $3\text{-}QI(11,2)$, and define 
     \[ \mathcal{F}_1=\{B : |A\cap B|=3,\ |B|=4\}, \quad \mathcal{F}_2=\{B : |A\cap B|=4,\ |B|=4\}.\] 
     Then $|\mathcal{F}_1 \cup \mathcal{F}_2| \le 8$ within any colour class containing $A$.
     \label{comb2cd}
     \end{lemma}
     \begin{proof}
     Each vertex in $\mathcal{F}_1$ contains a $3$-subset of $A=abcde$, while each vertex in $\mathcal{F}_2$ is a $4$-subset of $A$. By Condition~\ref{adjacencycondition}, two vertices in $\mathcal{F}_2$ must intersect in exactly three digits to lie in the same colour class. Similarly, any two vertices in $\mathcal{F}_1$ must intersect in either one or three digits. \\
     
     For a fixed $3$-subset of $A$, there are exactly two vertices in $\mathcal{F}_2$ containing it. If all vertices in $\mathcal{F}_1$ correspond to a fixed $3$-subset of $A$, then $\mathcal{F}_1$ contributes at most six vertices (one for each element of $\overline{A}=fghijk$), and $\mathcal{F}_2$ contributes at most two vertices containing the same $3$-subset. Thus, $|\mathcal{F}_1 \cup \mathcal{F}_2|\le 8$, and equality is achieved, for example, by \[\{abcf, abcg, abch, abci, abcj, abck, abcd, abce\}.\]
     On the other hand, if $\mathcal{F}_1$ contains vertices corresponding to two distinct $3$-subsets of $A$, then these impose certain constraints. Indeed, for any two distinct $3$-subsets of $A$, there exists a $4$-subset of $A$ among the respective two options from $\mathcal{F}_2$, containing exactly one of them. Consequently, the admissible vertices in $\mathcal{F}_2$ are reduced to at most one, and may in fact be eliminated entirely. In this case, $\mathcal{F}_1$ alone can contribute at most eight vertices. For instance, \[ \{abcf, adeg, adeh, adei, adej, adek, bcdf, bcef\}\] 
     is a valid colour class consisting entirely of vertices from $\mathcal{F}_1$. Thus, in all cases, $|\mathcal{F}_1 \cup \mathcal{F}_2|\le 8$.
\end{proof}
     
     Lemma~\ref{comb1bc} shows that Cases~1a and~1b together contribute at most seven vertices to a colour class containing $A$ in $3\text{-}QI(11,2)$, while Lemma~\ref{comb2cd} shows that Cases~2c and~2d contribute at most eight. Thus, the upper bound on the size of a colour class containing $A$ reduces to 
     \[1+7 (\text{Cases 1a and 1b})+6+5+10+8 (\text{Cases 2c and 2d)}=37.\]
     We further refine this bound by analysing the constraints imposed by the choices in Case~2b. In Lemma~\ref{qill2}, we show that if a colour class does not contain any vertex from Case~2b, then its size is at most $26$. On the other hand, suppose the colour class contains a vertex from Case~2b. Without loss of generality, let $B_1 = afgh$ be such a vertex, where $A = abcde$. This choice further reduces the admissible vertices in Case~2a (the $4$-subsets of $\overline{A} = fghijk$) from five to at most three, namely, either $\{fghi, fghj, fghk\}$ or $\{fijk, gijk, hijk\}$. Consequently, the upper bound on the size of a colour class reduces to
     \[1+7 (\text{Cases 1a and 1b})+6+3+10+8 (\text{Cases 2c and 2d)}=35.\]
     In the following lemma, we show that, under certain conditions, this upper bound reduces to $26$. \\
   
 \begin{lemma}
   Let $A = abcde$ be a vertex of $3\text{-}QI(11,2) = (V,E)$, and let $\mathcal{F} = \{ B \in V : |A \cap B| = 1,\ |B| = 4 \}$. If a colour class $\mathcal{S}$ in $3\text{-}QI(11,2)$ that contains $A$ does not contain any vertex from $\mathcal{F}$, then $|\mathcal{S}|\leq 26$.
   \label{qill2}
  \end{lemma}

\begin{proof}
    Since $\mathcal{S}$ does not contain any vertex from $\mathcal{F}$ (that is, from Case 2b in Section~\ref{sec4}), we have  $|\mathcal{S}|\leq 1(abcde)+7(\text{Cases 1a and 1b, Lemma \ref{comb1bc}})+6(\text{Case 1c})+5(\text{Case 2a})+8(\text{Cases 2c and 2d,}$ $\text{Lemma \ref{comb2cd}})=27$. \\
    
    Furthermore, if $\mathcal{S}$ contains no vertex from Case 2c, then $|\mathcal{S}|\leq 1(abcde)+7(\text{Cases 1a}$ and 1b) $+6(\text{Case 1c})+5(\text{Case 2a})+5(\text{Cases 2c and 2d})=24$. Therefore, without loss of generality, suppose that $B=abcf$ is chosen from Case 2c. This choice restricts the possible vertices from Case 1b to those listed in the table in Figure~\ref{none}, where each entry corresponds to a vertex $X$ satisfying the indicated intersection conditions with $A$ and $B$. By Lemma~\ref{condition}, no two vertices within a row can belong to the same hyperedge; hence, the vertices within a row can be assigned the same colour.\\
    
    \begin{figure}[h]
\begin{center}
    \begin{tabular}{|c|c|c|c|c|c|}
    \hline
        \multirow{2}{*}{$|X\cap A|=1$, $|X\cap B|=0$} & 
        $dghij$ & $dghik$ &$dgijk$&{\cellcolor{gray!40}$dhijk$}&$dghjk$ \\ \hhline{|~|-|-|-|-|-|}
         &$eghij$ & $eghik$ &$egijk$&\cellcolor{gray!40}$ehijk$&$eghjk$ \\     \hline
         \multirow{4}{*}{\parbox{4cm}{$|X\cap A|=1$, $|X\cap B|=1$ and $X\cap A\neq X\cap B$} } &
         $dfghi$&$dfghj$&$dfghk$&$dfgij$&$dfgik$\\   \hhline{|~|-|-|-|-|-|}
         &$dfgjk$&\cellcolor{gray!40}$dfhij$&\cellcolor{gray!40}$dfhik$&\cellcolor{gray!40}$dfhjk$&\cellcolor{gray!40}$dfijk$\\    \hhline{|~|-|-|-|-|-|}
         &$efghi$&$efghj$&$efghk$&$efgij$&$efgik$\\   \hhline{|~|-|-|-|-|-|}
         &$efgjk$&\cellcolor{gray!40}$efhij$&\cellcolor{gray!40}$efhik$&\cellcolor{gray!40}$efhjk$&\cellcolor{gray!40}$efijk$\\    \hline
        \multirow{3}{*}{\parbox{4cm}{$|X\cap A|=1$, $|X\cap B|=1$ and $X\cap A= X\cap B$}}& $aghij$ & $aghik$ &$agijk$&\cellcolor{gray!40}$ahijk$&$aghjk$ \\    \hhline{|~|-|-|-|-|-|}
         &\cellcolor{gray!40}$bghij$ & \cellcolor{gray!40}$bghik$ &\cellcolor{gray!40}$bgijk$&\cellcolor{gray!40}$bhijk$&\cellcolor{gray!40}$bghjk$ \\    \hhline{|~|-|-|-|-|-|}
         
         &\cellcolor{gray!40}$cghij$ & \cellcolor{gray!40}$cghik$ &\cellcolor{gray!40}$cgijk$&\cellcolor{gray!40}$chijk$&\cellcolor{gray!40}$cghjk$ \\    \hhline{|-|-|-|-|-|-|}
    \end{tabular}\caption{Vertices from Case 1b that can belong to a colour class with $A=abcde$ and $B=abcf$.}\label{none}\end{center}
    \end{figure}

    Since Case 2d contributes at most five vertices, a contribution of at most two vertices from Case 2c implies that Cases 2c and 2d together contribute at most 7, yielding $|\mathcal{S}|\leq 26$. If Case 2c contributes at least three vertices, then two possibilities arise:
    \begin{enumerate}
         \item[(i)] Case 2c contributes only vertices whose intersection with $B$ is exactly $abc$.
         \item[(ii)] Case 2c contributes at least one vertex whose intersection with $B$ is not equal to $abc$.
    \end{enumerate}
    
In Case (i), besides $B=abcf$, at most five vertices of the form $abcx$, where $x\in\{g,h,i,j,k\}$, can be included in $\mathcal{S}$. If any two such vertices are chosen, the last three rows of Figure~\ref{none} are eliminated, and hence no vertex containing $a, b,$ or $c$ can be included. By Lemma~\ref{comb1bc}, achieving seven vertices from Cases 1a and 1b requires all vertices corresponding to a fixed 4-subset of $fghijk$, which necessarily includes vertices containing $a, b,$ and $c$. Thus, the combined contribution from Cases 1a and 1b is at most 6. Therefore, $|\mathcal{S}|\leq 26$.\\

 In Case (ii), suppose that, in addition to $B=abcf$, a vertex $B_1$
 from Case 2c with a different 3-subset of $A$ is included. If this 3-subset intersects $abc$ in exactly one element, say $ade$, then $f\notin B_1$. Without loss of generality, let $B_1=adeg$. In this case, the table in Figure~\ref{none} reduces to the shaded entries and none from Case 2d is compatible. By Lemma~\ref{comb1bc}, the only subset of this that can yield the maximum contribution from Cases 1a and 1b is \[R=\{ahijk, bhijk, chijk, dhijk, ehijk\}.\] Further, Case 2c can contribute at most five vertices compatible with $B$ and $B_1$, namely, $adeg, adeh, adei, adej, adek$. Moreover, for $x\in\{g,h,i,j,k\}$, vertices of the form $abcx$ and $adex$ cannot simultaneously belong to the same colour class. Including any additional vertex of the form $abcx$ or $adex$ with $x\neq g$ eliminates at least one vertex from $R$. Consequently, the combined contribution from Cases 1a and 1b is at most 6 and the total size of $\mathcal{S}$ remains at most 26.\\

 Finally, suppose in addition to $B=abcf$, a vertex $B_1$
 from Case 2c with a different 3-subset of $A$, intersects $abc$ in two elements is selected; then it must contain $f$, and without loss of generality let $B_1= abdf$. In this case, the remaining vertices from Case 2c are $abef$, $bcdf$, $acdf$, $cdeg$, $cdeh$, $cdei$, $cdej$ and $cdek$. If two vertices with the 3-subset $cde$ are chosen, then only the seventh and eighth rows of Figure~\ref{none} remain, which cannot realize the optimal configuration of seven vertices from Cases 1a and 1b required by Lemma~\ref{comb1bc}, and hence $|\mathcal{S}|\leq 26$. Thus, assume that at most one vertex with 3-subset $cde$ is chosen. Moreover, the choice of $abef$ eliminates both $bcdf$ and $acdf$, and each selection from Case 2c restricts Case 2d to at most two compatible vertices. Consequently, the combined contribution from Cases 2c and 2d is at most 7 rather than 8, and hence, $|\mathcal{S}|\leq 26$. 
\end{proof}

    \begin{lemma}
     Let $A = abcde$ be a vertex of $3\text{-}QI(11,2) = (V,E)$, and let 
     \(\mathcal{F} = \{ B \in V : |A \cap B| = 1,\ |B| = 4 \}.\)  
     If a colour class $\mathcal{S}$ in $3\text{-}QI(11,2)$ that contains $A$ includes vertices $B_1, B_2 \in \mathcal{F}$ such that $|B_1 \cap B_2| = 3$ and $A \cap B_1 \neq A \cap B_2$, then $|\mathcal{S}| \leq 26$.
     \label{qill1}
     \end{lemma}
     \begin{proof}
    Since $A \cap B_1 \neq A \cap B_2$, the vertices $B_1$ and $B_2$ must contain the same $3$-subset of $\overline{A}$. Without loss of generality, let $B_1 = afgh$ and $B_2 = bfgh$. This choice invalidates the first three rows of the array in Figure~\ref{14_1table} (see Lemma~\ref{14_3}) and restricts the admissible vertices in Case~1c. In particular, Case~1c can now contribute at most $5$ vertices instead of $6$. These arise either from one of the last three rows (corresponding to $4$-subsets of $A$ together with a fixed element from $\{i,j,k\}$), or from at most three vertices in a column (corresponding to a fixed $4$-subset of $A$ together with elements from $\{i,j,k\}$). Therefore, 
    \[|\mathcal{S}| \le 1 + 7 \;(\text{Cases~1a, 1b}) + 5 + 3 + 10 \;(\mathcal{F}) + 8 \;(\text{Cases~2c, 2d}) = 34.\]
    If $|\mathcal{S} \cap \mathcal{F}| \le 2$, then \(|\mathcal{S}| \le 1 + 7 + 5 + 3 + 2 \;(\mathcal{F}) + 8 = 26.\) Hence, in the remainder of the proof, assume that $|\mathcal{S} \cap \mathcal{F}| \ge 3$. Suppose, in addition to $B_1 = afgh$ and $B_2 = bfgh$, there exists a third vertex $B_3 \in \mathcal{S} \cap \mathcal{F}$. Then there are three possible cases, according to whether $|B_i \cap B_3| \in \{0,1,3\}$ for $i = 1,2$.\\

    If $|B_i \cap B_3| = 3$ for $i = 1,2$, then without loss of generality, let $B_3 = cfgh$. This choice eliminates all vertices in Case~2c that involve any of the elements $f, g, h$ as such vertices intersect one of the $B_i$, $i=1, 2, 3,$ at two digits, which is not permitted. Furthermore, the selection of a vertex from Case~2c restricts the admissible vertices in Case~2d to at most two. If an additional vertex from Case~2c, corresponding to a different $3$-subset of $A$, is included, then at most one vertex remains admissible in Case~2d. Therefore, in all cases, the combined contribution from Cases~2c and~2d is at most five. Consequently, \[
|\mathcal{S}| \le 1 + 7 \;(\text{Cases~1a, 1b}) + 5 + 3 + 10 + 5 \;(\text{Cases~2c, 2d}) = 31.\] 
With the above choice of $B_1, B_2,$ and $B_3$, the remaining options from Case~2b are:
\begin{center}
\begin{tabular}{ccccccccc}
    $dfgh$ & $efgh$ &&&&&&&\\
    $aijk$ & $bijk$ & $cijk$ & $dijk$ & $eijk$ &&&&\\
    $dfij$ & $dfik$ & $dfjk$ & $dgij$ & $dgik$ & $dgjk$ & $dhij$ & $dhik$ & $dhjk$\\
    $efij$ & $efik$ & $efjk$ & $egij$ & $egik$ & $egjk$ & $ehij$ & $ehik$ & $ehjk$
\end{tabular}
\end{center}Note that selecting $dfgh$ and $efgh$, from the first row in the above array, makes the third and fourth rows invalid, respectively. On the other hand, each vertex in the second row intersects every $5$-subset of $\overline{A}$ in exactly two elements. Consequently, choosing any vertex from the second row forces the exclusion of all the vertices from Case~1a. Similarly, choosing any vertex from the second row reduces the choices from Case~1c to at most three.  Moreover, individually, the third or fourth row can contribute at most three vertices, whereas together they can contribute at most four. For example, without loss of generality, if $B_4 = dfij$ is selected from the third row, then the only remaining valid vertices in the third row are either $\{dfik, dfjk\}$ or $\{dgij, dhij\}$. The vertices from the fourth row that remain compatible with $B_4$ are \(efij,\ egik,\ egjk,\ ehik,\ ehjk.\) Each additional selection from the third row eliminates at least three of these remaining vertices from the fourth row. Hence, the combined contribution from the third and fourth rows is at most four. We now proceed by considering the number of vertices selected from the second row. 
     \begin{enumerate}
         \item[(i)] If at least two vertices are selected from the second row, then all vertices in the remaining last three rows of the array in Figure~\ref{14_1table} (see Lemma~\ref{14_3}) are eliminated, and Case~1c contributes no vertices. Thus,
         \[|\mathcal{S}| \leq 1 + 0 \;(\text{Case~1a}) + 5 \;(\text{Case~1b}) + 0 \;(\text{Case~1c}) + 3 + 10 + 5 = 24.\] 
         \item[(ii)] If exactly one vertex is selected from the second row, then Case~2b can contribute at most eight vertices in total. Indeed, in addition to $B_1, B_2,$ and $B_3$, the possibilities are as follows: either two vertices from the first row and one from the second row; or one vertex each from the first and second rows together with three vertices from either the third or fourth row; or no vertex from the first row, one from the second row, and at most four vertices from the third and fourth rows combined. Hence, \[|\mathcal{S}| \leq 1 + 0 \;(\text{Case~1a}) + 5 \;(\text{Case~1b}) + 3 \;(\text{Case~1c}) + 3 + 8 + 5 = 25.\] 
         \item[(iii)] If no vertex is selected from the second row, then Case~2b can contribute at most seven vertices in total. Indeed, in addition to $B_1, B_2,$ and $B_3$, the possibilities are as follows: either two vertices from the first row and none from the remaining rows; or one vertex from the first row together with three vertices from either the third or fourth row; or no vertex from the first row and at most four vertices from the third and fourth rows combined. Therefore, 
         \[|\mathcal{S}| \leq 1 + 7 \;(\text{Cases~1a, 1b}) + 5 \;(\text{Case~1c}) + 3 \;(\text{Case~2a}) + 5 \;(\text{Case~2b}) + 5 = 26,\] 
         when no vertex is selected from the third and fourth rows. On the other hand, selecting a vertex from either the third or fourth row reduces the number of admissible vertices in Case~2a to at most one. Hence, 
         \[|\mathcal{S}| \leq 1 + 7 \;(\text{Cases~1a, 1b}) + 5 \;(\text{Case~1c}) + 1 \;(\text{Case~2a}) + 7 \;(\text{Case~2b}) + 5 = 26.\]
     \end{enumerate}
     Thus, if $|B_i \cap B_3| = 3$ for $i = 1,2$, then $|\mathcal{S}| \leq 26$. Therefore, in the remainder of the proof, we may assume that $\mathcal{S}$ contains no vertex from the set $\{cfgh, dfgh, efgh\}$. \\

    If $|B_i \cap B_3| = 1$ for $i = 1,2$, then without loss of generality, let $B_3 = cfij$. With this choice of $B_1, B_2$ and $B_3$, Case~1a can contribute only $fghij$ and from Case~1b, the following vertices remain:
\begin{center}
\begin{tabular}{cccccc}
    $afghk$ & $bfghk$ & $dfghk$ & $efghk$ \\
    $cfghi$ & $cfijk$ & $cgijk$ & $chijk$ & $cfghj$
\end{tabular}
\end{center}
From this collection, at most four vertices can belong to $\mathcal{S}$. Moreover, selecting $fghij$ from Case~1a eliminates all vertices in the first row and leaves at most two admissible vertices from the second row. Hence, combining Cases~1a and~1b yields at most four vertices in total, with the maximum attained by selecting vertices from the first row above. Similarly, Case~2a can contribute only $fijk$ or $fghk$ and  the remaining options from Case~2c are given in Table~\ref{tabb}.
\begin{table}[h]
\centering
\begin{tabular}{|c|c|c|c|c|c|c|c|c|c|}
\hline
    \cellcolor{gray!40}$cdeg$ & \cellcolor{gray!40}$cdeh$ &&&&&&&&\\\hline
    \cellcolor{gray!40}$abdi$ & \cellcolor{gray!40}$abei$ & $adei$ & $bdei$ &&&&&&\\\hline
    \cellcolor{gray!40}$abdj$ & \cellcolor{gray!40}$abej$ & $adej$ & $bdej$ &&&&&& \\\hline
    \cellcolor{gray!40}$abdk$ & \cellcolor{gray!40}$abek$ & $adek$ & $bdek$ & \cellcolor{gray!40}$abck$ & $acdk$ & $acek$ & $bcdk$ & $bcek$ & \cellcolor{gray!40}$cdek$\\\hline
\end{tabular}
\caption{Options remaining in Case~2c when $B_3 = cfij$}
\label{tabb}
\end{table}
From each row of Table~\ref{tabb}, at most four vertices can be selected (see the proof of Lemma~\ref{14_6}). Moreover, selecting a vertex from the second row leaves only one admissible vertex in the third row, and vice versa. Selecting two vertices from the second row eliminates all vertices from the third row, and vice versa. If either $cdeg$ or $cdeh$ is selected, then Table~\ref{tabb} reduces to only shaded entries, from which at most five can be in $\mathcal{S}$ (for instance, $cdeg, cdeh, cdek, abdi, abdj$). On the other hand, if neither $cdeg$ nor $cdeh$ is selected, then choosing one vertex from either the second or third row leaves four admissible vertices in the fourth row, of which at most three can be selected. Choosing two vertices from the second or third rows leaves only one admissible vertex in the fourth row. In either case, the total number of vertices selected from Table~\ref{tabb} is at most five. \\

Furthermore, each selection from Table~\ref{tabb} leaves at most two admissible vertices in Case~2d. A subsequent selection corresponding to a different $3$-subset of $A$ reduces the number of admissible vertices in Case~2d by at least one. Since there are at most three vertices in Table~\ref{tabb} corresponding to a fixed $3$-subset of $A$, the combined contribution from Cases~2c and~2d is at most five. Hence, in this case $|\mathcal{S}|\leq 1+ 4 \;(\text{Cases 1a, 1b})+5\;(\text{Case 1c})+1\;(\text{Case 2a})+ 10\;(\text{Case 2b})+5\;(\text{Cases 2c, 2d})=26.$ \\

 Finally, we assume that Case~2b does not contribute any vertex $B_3$ such that $|B_i\cap B_3|=3$ or $|B_i\cap B_3|=1$ for $i=1, 2$. Thus, if $|B_i \cap B_3| = 0$ for $i = 1,2$, then without loss of generality, let $B_3 = cijk$. Consequently, Case~1c can contribute only three vertices, namely, $abdei, abdej, abdek$. Similarly, Case~2a can contribute at most three vertices, either $\{fijk, gijk, hijk\}$ or $\{fghi, fghj, fghk\}$. Moreover, the remaining vertices from Case~1b are shown in Table~\ref{tab:my_labe}.\\
\begin{table}[h]
    \centering
    \begin{tabular}{|c|c|c|c|c|c|}
    \hline
        $afghi$ & $afghj$ & $afghk$ & & & \\
        \hline
         $bfghi$& $bfghj$ & $bfghk$ &&&\\
         \hline
         &&&$cfijk$&$cgijk$&$chijk$\\
         \hline
         $dfghi$&$dfghj$&$dfghk$&$dfijk$&$dgijk$&$dhijk$\\
         \hline
         $efghi$&$efghj$&$efghk$&$efijk$&$egijk$&$ehijk$\\
         \hline
    \end{tabular}
    \caption{Choices from Case~1b, with $B_3=cijk$}
    \label{tab:my_labe}
\end{table}\\
 Note that any two vertices lying in different rows and different columns of Table~\ref{tab:my_labe} intersect in either two or three digits, and hence cannot both belong to $\mathcal{S}$. By Lemma~\ref{comb1bc}, for the union of Cases~1a and~1b to attain its maximum contribution, all five vertices corresponding to a fixed $4$-subset of $\overline{A}$ from Case~1b must be available. However, this condition is not satisfied in the above array. Therefore, the combined contribution of Cases~1a  and~1b is at most six vertices in $\mathcal{S}$. Further, with the above choice of $B_1, B_2,$ and $B_3$, the remaining valid options in  Case~2b are $aijk, bijk, dijk,$ and $ eijk$. Thus, Case~2b can contribute at most seven vertices in $\mathcal{S}$. Moreover, any such choice from Case~2b eliminates all possibilities from Case~1c. Therefore, Case 1c and 2b combined can contribute at most seven vertices. Consequently, $|\mathcal{S}|\leq 1+ 6\,(\text{Case 1a, 1b})+7\,(\text{Case~1c, 2b})+ 3\,(\text{Case~2a})+8\,(\text{Case~2c, 2d)}=25$.  

\end{proof}

 
  \begin{theorem}\label{QI}  
  The hypergraph $3\text{-}QI(11,2)$ has strong independence number $\alpha_S(3\text{-}QI(11,2))$ $=26$ and strong chromatic number $\chi_S(3\text{-}QI(11,2))\geq31$.
  \end{theorem}
  
  \begin{proof}

Let $A=abcde$ be a vertex of $3\text{-}QI(11,2)$ and let $\mathcal{S}$ be a colour class containing $A$. If $\mathcal{S}$ contains no vertex from Case~2b, then $|\mathcal{S}|\leq 26$ by Lemma~\ref{qill2}. If $\mathcal{S}$ contains two vertices $B_1,B_2$ from Case~2b with $|B_1\cap B_2|=3$ and $A\cap B_1\neq A\cap B_2$, then $|\mathcal{S}|\leq 26$ by Lemma~\ref{qill1}. Thus, in the rest of the proof, we may assume that $\mathcal{S}$ contains at least one vertex from Case~2b, but no pair $B_1, B_2$ satisfying the above condition. Without loss of generality, let $B_1=afgh$. If Case~2b contributes only $B_1$, then 
\[|\mathcal{S}|\le 1+7\,\text{(Case 1a, 1b)} + 6 + 3\,\text{(Case 2a)}+ 1\,\text{(Case 2b)}+8\,\text{(Case 2c, 2d)}=26. \]
If Case~2b contributes at least two vertices, then any additional vertex $B_2\neq B_1$ must satisfy one of the following: either $|B_1\cap B_2|=0$, or $|B_1\cap B_2|=1$, or $|B_1\cap B_2|=3$ with $A\cap B_1 = A\cap B_2$. With the choice $B_1=afgh$, the remaining admissible vertices in Case~2b are given by
\begin{center}
    \begin{tabular}{cccccccccccc}
      $aijk$ &  $bijk$ & $cijk$ & $dijk$ & $eijk$\\
      $bfij$ & $bfik$ & $bfjk$ & $bgij$ & $bgik$ & $bgjk$ & $bhij$ & $bhik$ & $bhjk$ \\
      $cfij$ & $cfik$ & $cfjk$ & $cgij$ & $cgik$ & $cgjk$ & $chij$ & $chik$ & $chjk$ \\
      $dfij$ & $dfik$ & $dfjk$ & $dgij$ & $dgik$ & $dgjk$ & $dhij$ & $dhik$ & $dhjk$ \\
      $efij$ & $efik$ & $efjk$ & $egij$ & $egik$ & $egjk$ & $ehij$ & $ehik$ & $ehjk$ \\
      $afgk$ & $afgj$ & $afgi$ &  $aghk$ & $aghj$ & $aghi$ & $afhk$ & $afhj$ & $afhi$ 
    \end{tabular}
\end{center}
Note that each of the last five rows can contribute at most three vertices. If more than one vertex is selected from the first row, then the result follows from Lemma~\ref{qill1}. Therefore, we assume that the first row contributes at most one vertex. If the vertex $aijk$ from the first row is selected, then no vertex from the remaining five rows can be chosen, as each would intersect $aijk$ in two digits. Moreover, in this case, the combined contribution from Cases~1a and~1b is strictly less than $7$. Indeed, by Lemma~\ref{comb1bc}, attaining the maximum requires that all seven vertices corresponding to a fixed $4$-subset of $\overline{A}$ be available; however, any such choice invalidates either $B_1$ or $B_2$. Thus, 
\[|\mathcal{S}|\le 1+6\,\text{(Case 1a, 1b)} + 6 + 3\,\text{(Case 2a)}+ 2\,\text{(Case 2b)}+8\,\text{(Case 2c, 2d)}=26. \]
If a vertex different from $aijk$ is selected from the first row, then without loss of generality let it be $B_2=bijk$. In this case, only the second and the last rows of the above array remain admissible in Case~2b, since every vertex in rows $3$--$5$ intersects $B_2$ in two digits. Thus, in addition to $B_1$ and $B_2$, Case~2b contributes at most six further vertices. However, if any such vertex from the second or the last row is selected, then Case~2a contributes at most one vertex. Hence, Cases~2a and~2b together contribute at most nine vertices. Further, with the choice of $B_1$ and $B_2$, Case~1a contributes no vertex, while Case~1b reduces to the vertices $afghx, cfghx, dfghx, efghx$ for $x\in\{i,j,k\}$ and $byijk, cyijk, dyijk, eyijk$ for $y\in\{f,g,h\}$. Moreover, Case~1c reduces to $acdei, acdej, acdek, bcdef, bcdeg, bcdeh$, of which at most three can coexist. Thus,
\[
|\mathcal{S}|\le 1 + 4\,(\text{Cases~1a,~1b}) + 3\,(\text{Case~1c}) + 9\,(\text{Case~2a, 2b}) + 8\,(\text{Cases~2c,~2d}) = 25.
\]

If no vertex is selected from the first row, then either a vertex is selected from rows $2$--$5$, or none is selected from these rows. Suppose first that a vertex is selected from rows $2$--$5$. Without loss of generality, let $B_2=bfij$ (the case $|B_1\cap B_2|=1$). In this case, from Case~1a only the vertex $fghij$ remains. From Case~1b, the admissible vertices are
\[
afghk,\ cfghk,\ dfghk,\ efghk,\ afgij,\ afhij,\ aghik,\ aghjk,\ bfghi,\ bfghj,\ bfijk.
\]
However, at most four of these can coexist in a colour class. Indeed, attaining the maximum requires, by Lemma \ref{comb1bc}, that all vertices corresponding to a fixed $4$-subset of $\overline{A}$ be present, which is not possible here. Hence, the combined contribution from Cases~1a and~1b is at most four. Similarly, Case~1c contributes at most five vertices, namely $abcdk, abcek, abdek,$ $ acdek,$ and $bcdek$. From Case~2a, only $fijk$ or $fghk$ remains admissible. If a vertex $B_3$ from rows $3$--$6$ is selected with $|B_2\cap B_3|=3$, then the case reduces to Lemma~\ref{qill1}. Otherwise, suppose $B_3$ from Case~2b satisfies $|B_2\cap B_3|=1$; without loss of generality, let $B_3=cgik$. In this case, Case~1c is further restricted and contributes at most one vertex from $\{bcdeh,\, acdej,\, abdek\}$. Hence,
\[
|\mathcal{S}|\le 1 + 4\,(\text{Cases~1a,~1b}) + 1\,(\text{Case~1c}) + 1\,(\text{Case~2a}) + 10\,(\text{Case~2b}) + 8 = 25.
\]

\noindent If $|B_2\cap B_3|=0$, then either $B_3=cijk$, which has already been excluded with the first row, or $B_3$ intersects $B_1$ in two elements, and hence is not admissible. If none of the above choices for $B_3$ occur, then Case~2b contributes only vertices from the second and the last rows, apart from $B_1$, yielding at most seven vertices. Consequently,
\[
|\mathcal{S}|\le 1 + 4\,(\text{Cases~1a,~1b}) + 5\,(\text{Case~1c}) + 1\,(\text{Case~2a}) + 7\,(\text{Case~2b}) + 8 = 26.
\]

Finally, if only the last row contributes to $\mathcal{S}$, then without loss of generality let $B_2=afgi$. In this case, Case~2a contributes at most one vertex, namely, $fghi$ or $hijk$, and Case~2b contributes at most four vertices. From Cases~1a and~1b, the remaining admissible vertices are shown below. 
\begin{center}
    \begin{tabular}{ccccccc}
       $fghij$ & $fghik$ &$afghi$&$bfghi$&$cfghi$&$dfghi$&$efghi$ \\
       &&&$bhijk$ & $chijk$ &$dhijk$ &$ehijk$ \\
       $afghj$&$afghk$&$afgij$&$afgik$&$afgjk$&$afhik$&$aghik$
   \end{tabular}
\end{center}
To attain the maximum contribution from Cases~1a and~1b, all vertices corresponding to a fixed $4$-subset of $\overline{A}$ must be present. If any vertex of the form $xfghi$ with $x\in\{a,b,c,d,e\}$ is chosen, then Case~1c is restricted to at most five vertices, yielding
\[
|\mathcal{S}|\le 1 + 7\,(\text{Cases~1a,~1b}) + 5\,(\text{Case~1c}) + 1\,(\text{Case~2a}) + 4\,(\text{Case~2b}) + 8 = 26.
\]
Otherwise, if no such vertex is chosen, then Cases~1a and~1b together contribute at most five vertices, and hence
\[
|\mathcal{S}|\le 1 + 5\,(\text{Cases~1a,~1b}) + 6\,(\text{Case~1c}) + 1\,(\text{Case~2a}) + 4\,(\text{Case~2b}) + 8 = 25.
\]
Thus, in all cases, a colour class has size at most $26$. Since each colour class is a strongly independent set, it follows that \[\alpha_S(3\text{-}QI(11,2))\leq 26.\] 
Moreover, the following set of $26$ vertices forms a strongly independent set:
\[
\begin{array}{llllll}
12345 & 12346 & 12347 & 12348 & 12349 & 12340 \\
1234x & 1236  & 1237  & 1238  & 1239  & 1230 \\
123x  & 67890 & 6789x & 6790x & 6890x & 7890x \\
6780x & 6789  & 6780  & 6890  & 7890  & 6790 \\
1234  & 1235
\end{array}
\]
where $0$ and $x$ denote $10$ and $11$, respectively. Hence, $\alpha_S(3\text{-}QI(11,2))=26$, and consequently,
\[
\chi_S(3\text{-}QI(11,2)) \ge \left\lceil \frac{792}{26} \right\rceil = 31.
\]
\end{proof}

In \cite[Corollary~5]{raina}, it is shown that if $\chi_S(3\text{-}QI(11,2))>30$, then $CAN(3\text{-}QI(11,2),2)$ $=11$. Together with Theorem \ref{QI}, this yields the following.\\

\begin{corollary}\label{final}
    $CAN(3\text{-}QI(11,2),2)=11$.
\end{corollary}

Furthermore, if $H$ is a $3$-uniform hypergraph with $\chi_S(H)>30$, then there is no homomorphism $H \to 3\text{-}QI(10,2)$. By~\cite[Theorem 4.3]{raaphorst2018variable}, this implies that no covering array $CA(10,H,2)$ exists, and hence $CAN(H,2)\geq 11$. Moreover, if there exists a homomorphism $H \to 3\text{-}QI(11,2)$, then a covering array $CA(11,H,2)$ exists, and therefore $CAN(H,2)=11$.\\

\begin{corollary}
Let $H$ be a $3$-uniform hypergraph with $\omega_3(H)>5$. Then $CAN(H,2)>11$. Moreover, if there exists a homomorphism $H \to 3\text{-}QI(12,2)$, then $CAN(H,2)=12$.
\end{corollary}

\begin{proof}
From~\cite{raina}, $\omega_3(3\text{-}QI(11,2))=5$. Since homomorphisms preserve $3$-cliques, there is no homomorphism $H \to 3\text{-}QI(11,2)$, and hence, by Theorem~\ref{QI}, $CAN(H,2)>CAN(3\text{-}QI(11,2),$ $2)=11$. If there exists a homomorphism $H \to 3\text{-}QI(12,2)$, then a covering array $CA(12,H,2)$ exists, implying $CAN(H,2)=12$.
\end{proof}

\section{Conclusion}
\label{sec5}

Our results on the strong independence and strong chromatic numbers of $3\text{-}QI(11,2)$, by studying its vertices as families of intersecting set systems, imply that $CAN(3\text{-}QI(11,2),2)$ $=11$, providing further evidence towards the conjecture that $CAN(3\text{-}QI(n,2),2)=n$ for all $n\geq 8$. The conjecture is now verified for $n=8,9,10,11,12$~\cite{raina}; while this method can be extended to larger values of $n$ in theory, the large number of subcases arising from Condition \ref{adjacencycondition} renders it impractical.

\section*{Acknowledgements}
The first author is grateful to the Department of Science and Technology (DST) for financial support under the DST-INSPIRE Senior Research Fellow scheme with sanction no. DST/INSPIRE Fellowship/2021/IF210377.

\bibliographystyle{ieeetr}
\bibliography{ref}

\end{document}